\documentclass{article}

\usepackage[a4paper,margin=1in]{geometry}

\usepackage{amsmath}          
\usepackage{amssymb}          
\usepackage{amsfonts}
\usepackage[normalem]{ulem}

\usepackage{amsthm}
\newtheorem{theorem}{Theorem}[section]

\usepackage{graphicx}
\usepackage{subcaption}
\usepackage{multirow}

\usepackage{tikz}
\usetikzlibrary{calc}
\usetikzlibrary{decorations.pathreplacing}

\usepackage{hyperref}
\usepackage{booktabs}

\usepackage{lipsum}
\usepackage{amsfonts}
\usepackage{graphicx}
\usepackage{epstopdf}
\usepackage{algorithmic}
\ifpdf
  \DeclareGraphicsExtensions{.eps,.pdf,.png,.jpg}
\else
  \DeclareGraphicsExtensions{.eps}
\fi

\title{QVaR: a Quantum Variational Regularization method for Linear Inverse Problems}

\author{
Siiri Rautio\thanks{Department of Mathematics and Information Science, Josai University, Japan (\texttt{srautio@josai.ac.jp}).}
\and
Hjørdis Schlüter\thanks{Department of Mathematics and Statistics, University of Helsinki, Finland.}
\and
Andreas Hauptmann\thanks{Research Unit of Mathematical Sciences, University of Oulu, Finland; and Department of Computer Science, University College London, UK.}
\and
Babak Maboudi Afkham\thanks{Research Unit of Mathematical Sciences, University of Oulu, Finland.}
}

\usepackage{amsopn}

\newcommand{\SR}[1]{{\color{red}{Siiri: #1}}}

\newcommand{\Real}{\mathbb{R}}

\DeclareMathOperator*{\argmin}{arg\,min}

\newcommand{\oper}[1]{\operatorname{\mathcal{#1}}}
\newcommand{\FwdOp}{A}
\newcommand{\signal}{f}
\newcommand{\data}{g}
\newcommand{\RecSpace}{X}
\newcommand{\DataSpace}{Y}

\newcommand{\RegFunc}{\oper{R}}
\newcommand{\RegParam}{\alpha}
\newcommand{\SparseTrafo}{B}
\newcommand{\Id}{I}

\begin{document}

\maketitle
\begin{abstract}
We present a tailored framework for solving regularized linear inverse problems using quantum optimization methods. By discretizing the solution space and encoding data fidelity and regularization terms into quadratic unconstrained binary optimization (QUBO) models, we formulate both Tikhonov- and sparsity-promoting regularized inverse problems within a unified quantum optimization framework. We further introduce a notion of quantum sensitivity that characterizes the effect of perturbations arising from approximate quantum evolution and discretization. We derive bounds relating these perturbations to stability properties of the underlying inverse problem, thereby establishing a theoretical connection between quantum solution variability and classical notions of ill-posedness. The framework is extended to variational inverse problems through wavelet-based representations and complemented by reduced-order modeling strategies in both parameter and Hamiltonian spaces to mitigate current hardware limitations. Numerical experiments on simulated and physical quantum hardware indicate that the resulting low-energy solution distributions retain information about the underlying inverse problem, while also revealing the limitations imposed by finite-time evolution, discretization, and hardware noise.
\end{abstract}






\section{Introduction}

Inverse problems play a central role in science and engineering, arising in applications ranging from medical imaging and geophysics to machine learning, and numerous industrial settings. In these problems, one seeks to infer unknown parameters from indirect and often noisy observations. Owing to their broad applicability, advances in the theory and methodology of inverse problems have the potential to significantly impact a wide range of scientific and engineering disciplines.

Inverse problems are often ill-posed in the sense of Jacques Hadamard \cite{hadamard2014lectures}, meaning that a solution may fail to exist, may not be unique, or may depend sensitively on perturbations in the data; see also the work by Nash on stability of Inverse Problems \cite{nash1958continuity}. Regularization provides a standard approach for addressing such ill-posedness by incorporating prior information into the problem formulation. Typically, regularization augments the objective with a penalty term that discourages solutions inconsistent with prior assumptions, such as smoothness or sparsity. This has the effect of stabilizing the problem and restricting the set of admissible solutions, thereby enabling the computation of meaningful and robust approximations.

Among inverse problems, linear inverse problems constitute one of the most widely studied classes due to their analytical structure and broad applicability \cite{mueller2012linear}. They arise in numerous important applications, including X-ray computed tomography (CT), magnetic resonance imaging (MRI), image deblurring across microscopic, macroscopic, and astronomical imaging, as well as seismic imaging \cite{kak2001principles,symes2009seismic}. 

Classical approaches cast linear inverse problems as regularized optimization problems \cite{engl1996regularization,hansen2010discrete,kirsch2011introduction,scherzer2009variational}, consisting of a data fidelity term, measuring the agreement between the model and the observed data, and a regularization term that encodes prior information about the solution. Common choices include Tikhonov regularization based on the $\ell^2$-norm, as well as sparsity-promoting regularization based on the $\ell^1$-norm, which has become particularly prominent in the context of compressed sensing \cite{donoho2006compressed}. Despite the maturity of the field, the development of new methodologies for linear inverse problems continues to attract significant attention, driven by emerging applications, large-scale data, and the need for improved robustness and interpretability.  

In parallel, recent years have witnessed substantial progress in quantum computing technologies, with modern devices supporting systems of tens to hundreds of qubits. While current hardware remains limited in scale and reliability, these developments have motivated increasing interest in quantum algorithms for optimization and inference tasks. In particular, quantum annealing and related approaches provide a framework for solving quadratic unconstrained binary optimization (QUBO) problems, which offer a unifying representation for a broad class of discrete optimization problems \cite{Lucas2014}. In this setting, the optimization objective is encoded into a Hamiltonian whose low-energy configurations correspond to candidate solutions. Quantum optimization methods then aim to evolve the system toward such low-energy states, so that approximate solutions to the original optimization problem can be obtained from measurements of the final quantum state \cite{Nannicini25}.


This work investigates the formulation of linear inverse problems within the QUBO framework. By discretizing the solution space and encoding regularized objectives, such as Tikhonov and $\ell^1$ penalties, into binary quadratic forms, we enable the use of quantum optimization methods for inverse problem solving. This shifts the focus from designing specialized optimization algorithms for each regularization term to designing the cost function itself, often a simpler and more flexible task.

A key challenge in inverse problems is the characterization of uncertainty. Traditional uncertainty quantification techniques, such as Bayesian sampling methods \cite{kaipio2005statistical}, are often computationally expensive for large-scale problems. In contrast, quantum optimization methods naturally produce distributions over low-energy solutions.
We investigate how these distributions, together with deviations arising from finite-time adiabatic evolution and discretization effects, may provide information about solution ambiguity, competing minimizers, and sensitivity to perturbations.

The proposed approach offers several potential advantages. It provides a unified framework in which diverse regularization strategies can be encoded directly into a QUBO model. Furthermore, quantum optimization methods naturally produce ensembles of low-energy configurations, which may facilitate the exploration of multimodal solution landscapes in nonconvex or severely ill-posed inverse problems. While such settings are beyond the scope of the present work, they represent a promising direction for future investigation. Recently, quantum computing methods have also been explored for discrete inverse problems on graphs~\cite{Ilmavirta2025}.

The contributions of this work are as follows. First, we present a systematic framework for formulating linear inverse problems with both Tikhonov and sparsity-promoting ($\ell^1$) regularization as quadratic unconstrained binary optimization (QUBO) problems, enabling their treatment within quantum optimization paradigms. Second, we introduce the notion of \emph{quantum sensitivity} and establish rigorous links between perturbations arising from quantum optimization and the stability properties of inverse problems, providing a theoretical interpretation of quantum solution variability. Third, we extend the proposed framework beyond finite-dimensional settings by considering variational formulations, in which solutions are represented in a wavelet basis, thereby bridging discrete and functional perspectives.

To address current hardware limitations, we develop reduced-order modeling strategies that operate both in the parameter space, through basis selection, and in the Hamiltonian representation, by discarding weak coupling terms. This enables the construction of low-dimensional approximations that remain computationally tractable on existing quantum devices. 
Finally, we investigate the proposed methodology through experiments on both simulated and physical quantum hardware, highlighting practical considerations and discrepancies between idealized and real-world implementations.

At the same time, important limitations remain. Current quantum hardware restricts the size of tractable problems, and embedding large-scale inverse problems into QUBO form can incur significant overhead. Moreover, discrepancies between simulated quantum algorithms and physical quantum devices can impact solution quality. Extending the proposed framework to nonlinear forward models or more complex regularization schemes requires higher-order Hamiltonian representations, which remain challenging to implement on existing platforms.

We present regularization methods for variational inverse problems in \ref{sec:variational-IP}. We recall the quantum optimization within the adiabatic theorem in \ref{sec:adiabatic} and introduce QUBO formulation of binary optimization in \ref{sec:QUBO}. The QVaR method is then introduced in \ref{sec:qvar} together with the notion of quantum sensitivity for linear inverse problems. We present numerical experiments of linear inverse problem with a $3\times2$ forward matrix and also a larger scale variational deblurring problem in \ref{sec:results}. Finally, we present conclusive remarks in \ref{sec:conclusion}.

\section{Variational Methods for Inverse Problems} \label{sec:variational-IP}
In this work, we consider inverse problems of the form
\begin{equation}\label{eqn:InvProb}
    \FwdOp\signal+\delta=\data,
\end{equation}
where $\FwdOp\colon\RecSpace\to\DataSpace$ is a linear operator that maps the unknown quantity of interest $\signal\in \RecSpace$ to the measured data $\data\in\DataSpace$. The measurements are assumed to be noise corrupted by noise $\delta$. Here we assume that both $\RecSpace$ and $\DataSpace$ are finite dimensional spaces, e.g., Euclidean spaces, in particular this means that $\FwdOp$ can be represented by a matrix and we assume $\delta\in Y$.

The inverse problem in \eqref{eqn:InvProb} is assumed to be ill-posed, that means given noisy data $\data$, the solution $\signal$ to \eqref{eqn:InvProb} may not exist, is non-unique, and 
does not depend continuously on the data \cite{hadamard2014lectures,nash1958continuity}. Thus, we need to define a well-posed solution operator \cite{engl1996regularization}, i.e., for each $\data$ we want to find a unique $\signal$ that approximately solves \eqref{eqn:InvProb} and depends continuously on the data $\data$. One of the most common ways to define such a solution operator is to search for solutions as the minimizer of a suitable cost function, which leads to the framework of variational methods \cite{scherzer2009variational}. 

In this paper, we aim to introduce quantum solution techniques to solve the variational problem given common regularizers, including classical smoothness as well as sparsity regularization for piecewise constant signals. The latter requires transforming the signal to a suitable basis, here given by the Haar wavelet transform, and imposing an $\ell^1$-penalty on the coefficients.  In the following, we will shortly review classical solution techniques, which often require an iterative solver, to contrast them later with the quantum variational methods.

\subsection{Smoothness and sparsity regularization}
 
Variational approaches to solve inverse problems search for a solution to the inverse problem \eqref{eqn:InvProb} as minimizer of a suitable variational cost function, such as 
\begin{equation}\label{eqn:VarProb}
    \signal^\star=\argmin_{\signal\in\RecSpace}\left\{\|\FwdOp\signal-\data\|^2_2 + \RegParam \RegFunc(\signal)\right\},
\end{equation}
where the first term measures data fidelity, $\RegFunc\colon\RecSpace\to\Real_+$ is a regularization functional, and $\RegParam>0$ controls the trade-off between the two terms.

The regularizer is essential to define a well-posed solution operator, if the inverse problem exhibits instability issues, e.g., when the singular values of $\FwdOp$ are rapidly decreasing. Furthermore, the choice of regularizer models prior assumptions that we have about the signal to be recovered. In fact, in a Bayesian formulation of inverse problems, the regularizer is connected to the prior distribution of the unknown $\signal$, which will be discussed in detail in Section \ref{sec:Uncertainty}. 

A simple choice of such a functional would be smoothness regularization, where $\RegFunc(\signal)=\|\signal\|_2^2$. This leads to a quadratic differentiable cost function. A closed form solution for $\signal^*$ is given by
$
\signal^*=\left(\FwdOp^T\FwdOp + \RegParam\Id\right)^{-1}\FwdOp^T \data.
$
Depending on the dimensionality of the matrix $A$, one may still need to solve above equation iteratively, for instance by conjugate gradient or other Krylov space methods \cite{brown2025inner,chung2024computational,gazzola2015krylov}.
Additionally, the quadratic problem can be naturally solved using quantum algorithms, as will be discussed in Section \ref{sec:qvar}.
Nevertheless, for many applications the assumption of smoothness on $\signal$ is not realistic, especially for applications of inverse problems in imaging.

Another popular choice is sparsity regularization, where signals are assumed to have only a few non-zero elements. This can be achieved by choosing the $\ell^1$-norm as regularizer, i.e., $\RegFunc(\signal)=\|\signal\|_1$. Unfortunately, in inverse problems the signal of interest is rarely sparse itself, but rather it is sparse with respect to a sparsifying transform $\SparseTrafo$, such that the regularizer becomes $\RegFunc(\signal)=\|\SparseTrafo\signal\|_1$.
For instance, a common assumption in practice are piecewise constant signals, which are often modeled by total-variation regularization \cite{rudin1992nonlinear}, i.e., sparsity of the gradient $\nabla\signal$. Here, we will follow a different route and use wavelet regularization. Specifically, we will make use of Haar wavelets for two reasons. First, Haar wavelet regularization is closely connected to total variation regularization \cite{steidl2004equivalence}. Second, the wavelet transform is invertible, which will be essential to formulate the proposed quantum variational methods.

In classical optimization, the sparsity regularized problem needs to be solved using techniques from convex analysis. Specifically, for the wavelet transform, we can make use of the successful iterative shrinkage-thresholding algorithm (ISTA) \cite{daubechies2004iterative}, which is an instance of a proximal gradient descent algorithm. 
However, it is possible to formulate the $\ell^1$ constrained problem to a suitable constrained quadratic programming problem, which will enable a suitable QUBO implementation we will discuss in Section \ref{sec:qvar}.


\section{Quantum Optimization and Adiabatic Theorem} \label{sec:adiabatic}

In this section, we introduce the basic principles of quantum computing and quantum optimization, with an emphasis on the QUBO formulation. We begin with a brief overview of the computational model underlying quantum algorithms and discuss how optimization problems may be encoded within this framework. We then recall the adiabatic theorem, which underpins a class of quantum optimization methods, and describe the corresponding QUBO-based approach for minimizing discrete binary objective functions. Finally, we show how regularized variational formulations arising in ill-posed inverse problems can be recast within the QUBO framework.

One of the main motivations for quantum computing is that a controllable quantum device can, in suitable settings, reproduce the unitary time evolution generated by a target Hamiltonian. By contrast, on a classical computer, simulating the dynamics of a many-body quantum system generally requires an explicit numerical approximation of the evolution operator, which can become prohibitively costly as the system size grows. This observation does not imply a blanket computational advantage for quantum computers, but it does suggest that quantum devices may offer a substantial advantage for some Hamiltonian-simulation tasks.

One approach to quantum optimization is to encode the objective function into a Hamiltonian whose ground state corresponds to the solution of interest, and to exploit quantum dynamics to prepare or approximate this ground state. A central theoretical result underlying such approaches is the quantum adiabatic theorem. For a quantum system governed by a time-dependent Hamiltonian with discrete energy levels, the theorem states that if the system is initially prepared in a non-degenerate eigenstate and the Hamiltonian is varied sufficiently slowly, while maintaining a non-vanishing spectral gap between this eigenvalue and the rest of the spectrum, then the system remains close to the corresponding instantaneous eigenstate throughout the evolution. This principle forms the basis of adiabatic quantum optimization methods.

In adiabatic quantum optimization, one begins with an initial Hamiltonian whose ground state can be prepared efficiently. The Hamiltonian is then varied gradually from this initial form to a final, problem-dependent Hamiltonian whose ground state encodes the desired minimizer. If the evolution is sufficiently slow and the instantaneous ground state remains separated from excited states by a nonvanishing spectral gap, then the evolving quantum state remains close to the instantaneous ground state throughout the process. Under these conditions, the final state approximates the ground state of the target Hamiltonian, and the solution to the optimization problem may be recovered from that final ground state.

Rigorous performance guarantees for adiabatic quantum optimization are available only under idealized assumptions. In the closed-system setting, the algorithm is analyzed in terms of an initial Hamiltonian, a final problem Hamiltonian, and a continuous interpolation between them; success depends on the properties of the entire evolution, rather than on the endpoints alone. In particular, the runtime may need to be very long when the minimum spectral gap along the interpolation is small. On realistic devices, noise and decoherence can perturb the intended Hamiltonian dynamics and induce transitions out of the instantaneous ground state. As a result, the output is generally probabilistic, and repeated runs are typically required to estimate the probability of obtaining a low-energy or optimal solution.

In what follows, we briefly recall the essentials of Hamiltonian dynamics for quantum states and then summarize the QUBO-based approach following \cite{Nannicini25}.

Quantum states evolve according to the Schr\"odinger equation
\begin{equation}\label{eq:schrodinger}
    i \hbar \frac{d}{dt} \psi(t) = \mathcal{H}(t)\,\psi(t).
\end{equation}
Here, $\psi(t)\in\mathcal{V}$ denotes the quantum state in a complex Hilbert space $\mathcal{V}$ with inner product $\langle\cdot,\cdot\rangle$ and norm $\|\cdot\|$. The (possibly time-dependent) Hamiltonian $\mathcal{H}(t):\mathcal{V}\to\mathcal{V}$ generates the dynamics, and $i$ is the imaginary unit.

We assume that the Hamiltonian $\mathcal{H}(t)$ is self-adjoint for each $t$. By the spectral theorem \cite{reed2012methods}, $\mathcal{H}(t)$ admits a spectral decomposition. In the case of a purely discrete spectrum, one can choose an orthonormal basis of eigenvectors $\{\phi_i(t)\}_{i=1}^\infty \subset \mathcal{V}$ with corresponding real eigenvalues or energy levels $\{\lambda_i(t)\}_{i=1}^\infty$, which we order (counting multiplicities) so that $\lambda_1(t) \leq \lambda_2(t) \leq \cdots$,
and
\[
\mathcal{H}(t)\phi_i(t) = \lambda_i(t)\,\phi_i(t), \qquad i=1,2,\dots,\qquad t>0.
\]
It is then convenient to expand the state $\psi(t)$ in this eigen-basis:
\[
\psi(t) = \sum_{i=1}^\infty \langle \phi_i(t), \psi(t) \rangle \,\phi_i(t).
\]
In this representation, the action of $\mathcal{H}(t)$ on $\psi(t)$ is given by
\[
\mathcal{H}(t)\psi(t) = \sum_{i=1}^\infty \langle \phi_i(t), \psi(t) \rangle \,\lambda_i(t)\,\phi_i(t).
\]
The ground state of the system is defined as an eigenvector associated with the smallest eigenvalue $\lambda_1(t)$. In equilibrium settings, states of lower energy are generally more robust to perturbations, and the ground state plays a distinguished role, for instance as the target of energy-minimization procedures in quantum optimization.

The adiabatic theorem, introduced below, provides a basis for quantum optimization via continuous deformation of Hamiltonians. Consider an interpolation between an initial Hamiltonian $\mathcal{H}_I$ and a final Hamiltonian $\mathcal{H}_F$. 
The Hamiltonian is then varied smoothly according to
\[
    \mathcal{H}(t) = (1 - s(t))\mathcal{H}_I + s(t)\mathcal{H}_F,
\]
where $s : [0,T] \to [0,1]$ is a monotone function with $s(0)=0$ and $s(T)=1$. 

Under suitable conditions, most notably, sufficiently slow evolution and the presence of a nonvanishing spectral gap between the ground state and the rest of the spectrum, the quantum state remains close to the instantaneous ground state of $\mathcal{H}(t)$ throughout the evolution. The motivation for this construction is that $\mathcal{H}_I$ is chosen so that its ground state is easy to prepare, while $\mathcal{H}_F$ is designed so that its ground state encodes the solution to the problem of interest.

The solution of the Schr\"odinger equation \eqref{eq:schrodinger} can be written formally as
\[
    \psi(t) = \mathcal{T} \exp\!\left( -\frac{i}{\hbar} \int_0^t \mathcal{H}(s)\, ds \right)\psi(0),
\]
where $\mathcal T$ denotes the time-ordering operator. In practice, a closed-form expression for the solution is generally not available. To approximate the evolution, we discretize the time interval $[0,T]$ into $N$ subintervals with step size $\Delta t = T/N$ and approximate the time integral using a quadrature rule, for example, left-point evaluation. This yields
\[
    \int_0^t \mathcal{H}(s)\, ds \;\approx\; \sum_{k=0}^{N-1} \mathcal{H}(t_k)\,\Delta t,
    \qquad t_k = k\Delta t.
\]
The corresponding time-ordered exponential can then be approximated by a product of short-time evolution operators,
\begin{equation} \label{eq:quantum-evolution}
    \psi(t) \;\approx\; \prod_{k=0}^{N-1} 
    \exp\!\left( -\frac{i}{\hbar}\mathcal{H}(t_k)\,\Delta t \right)\psi(0),
\end{equation}
so that factors with larger $k$ act first. Therefore, the resulting quantum evolution can be interpreted as a sequence of exponential operators applied successively to the initial ground state of $\mathcal{H}_I$. However, each factor involves the exponential of the interpolated Hamiltonian in \eqref{eq:quantum-evolution}, i.e., $(1 - s(t_k))\mathcal{H}_I + s(t_k)\mathcal{H}_F$,
which is generally difficult to implement directly. In particular, when $\mathcal{H}_I$ and $\mathcal{H}_F$ do not commute, the exponential of their sum cannot be expressed as a product of exponentials. To address this, one employs a first-order Trotter (Lie product) approximation \cite{reed2012methods}, which, for sufficiently small $\Delta t$, yields
\begin{equation} \label{eq:totter}
\begin{aligned}
\exp\!&\left( -\frac{i}{\hbar}\Delta t\left( (1 - s(t_k))\mathcal{H}_I + s(t_k)\mathcal{H}_F \right) \right)
\approx \\
&\exp\!\left( -\frac{i}{\hbar}\Delta t (1 - s(t_k))\mathcal{H}_I \right)
\exp\!\left( -\frac{i}{\hbar}\Delta t s(t_k)\mathcal{H}_F \right).
\end{aligned}
\end{equation}
This approximation enables the evolution to be implemented as a sequence of simpler unitary operations associated with $\mathcal{H}_I$ and $\mathcal{H}_F$. A central task in quantum optimization is therefore to design $\mathcal{H}_I$ and $\mathcal{H}_F$ so that these unitary operators can be further decomposed into elementary operations compatible with the gate set of a quantum computer. Such decompositions typically involve single-qubit rotations and two-qubit interaction gates, which together enable the implementation of the underlying dynamics. In the next section, we introduce the QUBO formulation and describe how such Hamiltonians can be constructed in terms of qubits, i.e., quantum state of a bit. 

To close this section, we present the adiabatic theorem based on Thm 9.11 in \cite{Nannicini25}, but for the purpose of this paper we restrict ourselves to the case of a non-degenerate ground state. We rescale time by setting $s=t/T \in [0,1]$. Then \eqref{eq:schrodinger} becomes
\[
i\hbar \frac{d}{ds}\psi(s) = T\,\mathcal{H}(s)\psi(s),
\]
and it is convenient to choose $\hbar=1$.
\begin{theorem}[Adiabatic theorem for a non-degenerate ground state] \label{thm:adiabatic}
Let $\mathcal{H}(s)$, $s\in[0,1]$, be a sufficiently smooth family of self-adjoint operators on a Hilbert space $\mathcal{V}$. Assume that for each $s$, $\mathcal{H}(s)$ has a simple smallest eigenvalue $\lambda_1(s)$ with normalized eigenvector $\phi_1(s)$, and that this eigenvalue is separated from the rest of the spectrum by a uniform gap $\gamma>0$.  If the system is initialized in the ground state $\psi(0)=\phi_1(0)$ and evolves according to
\[
i \frac{d}{ds}\psi(s) = T\,\mathcal{H}(s)\psi(s),
\]
then, provided $T$ is sufficiently large, the solution $\psi(1)$ remains close to the instantaneous ground state up to a phase factor. In particular, there exists a real-valued phase $\theta(s)$ such that
\[
\|\psi(s) - e^{i\theta(s)}\phi_1(s)\| \le \varepsilon, \qquad \varepsilon = \frac{ \sup_{s\in[0,1]}\| \partial_s \mathcal H(s) \|_{\mathrm{op}}}{T \gamma^2}, \qquad \|\mathcal A\|_{\mathrm{op}} := \sup_{\|u\|=1} \|\mathcal Au \|,
\]
for all $s \in [0,1]$. Consequently, the overlap with the ground state satisfies
\[
|\langle \phi_1(s), \psi(s) \rangle|^2 \ge 1 - \varepsilon^2.
\]
\end{theorem}

\section{Quadratic Unconstrained Binary Optimization (QUBO)} \label{sec:QUBO}
In this section, we introduce the QUBO formulation of a binary optimization and recast it as an eigenvalue problem within the adiabatic framework of the previous section. Our presentation follows standard constructions, e.g., \cite{Nannicini25}.


The objective of a QUBO is to solve
\begin{equation} \label{eq:QUBO}
    \begin{aligned}
    & \arg \min_{\boldsymbol{z} \in \{0,1\}^N} \; E(\boldsymbol{z}), \quad
    &E(\boldsymbol{z} ) = \boldsymbol{z}^\top Q \boldsymbol{z} + \boldsymbol{p}^\top \boldsymbol{z},
    \end{aligned}
\end{equation}
where $Q \in \mathbb{R}^{N \times N}$ and $\boldsymbol{p} \in \mathbb{R}^N$. Without loss of generality, $Q$ may be taken to be symmetric. Furthermore, $\boldsymbol{z}\in\{0,1\}^N$ is a binary vector.

To connect the QUBO formulation with a quantum representation, let each binary variable $z_i$, components of $\boldsymbol{z}$, belong to a two-dimensional complex vector space $\mathbb{C}^2$. The canonical basis vectors $e_0 = (1,0)^T$ and $e_1 = (0,1)^T$ represent the two possible states of this variable. For a system of $N$ binary variables, the state space is given by the tensor product $(\mathbb{C}^2)^{\otimes N}$, whose canonical basis consists of vectors of the form
\[
e_{\boldsymbol{z}} := e_{z_1} \otimes e_{z_2} \otimes \cdots \otimes e_{z_N}, 
\qquad z_i \in \{0,1\}.
\]
Each such basis vector corresponds uniquely to a binary configuration $\boldsymbol{z} \in \{0,1\}^N$.

In this representation, functions of the binary variables can be encoded as diagonal operators whose eigenvalues correspond to the objective evaluated at each configuration. To make this precise, consider the Hilbert space $\mathcal V = (\mathbb{C}^2)^{\otimes N}$ with canonical basis $\{e_{\boldsymbol{z}}\}_{\boldsymbol{z}\in\{0,1\}^N}$. We define the operator $\mathcal H_F : \mathcal V \to \mathcal V$ by
\begin{equation} \label{eq:final-hamiltonian}
    \mathcal H_F e_{\boldsymbol{z}} = E(\boldsymbol{z})\, e_{\boldsymbol{z}}, \qquad \boldsymbol{z}\in\{0,1\}^N,
\end{equation}
where $E$ is the objective function in \eqref{eq:QUBO}. By construction, $\mathcal H_F$ is a self-adjoint operator that is diagonal in this basis, with eigenvalues $\lambda_{\boldsymbol{z}} := E(\boldsymbol{z})$. In particular, minimizing the QUBO objective \eqref{eq:QUBO} is equivalent to finding the smallest eigenvalue of $\mathcal H_F$, and any corresponding eigenvector $e_{\boldsymbol{z}}$ represents a minimizer state.

Since $\mathcal H_F$ is diagonal in the computational basis $\{e_{\boldsymbol z}\}_{\boldsymbol z\in\{0,1\}^N}$, its exponential is diagonal in the same basis. In particular (cf. \eqref{eq:totter}),
\[
\exp\!\left(-i\,\Delta t\, s(t_k)\mathcal H_F\right)e_{\boldsymbol z}
=
\exp\!\left(-i\,\Delta t\, s(t_k)E(\boldsymbol z)\right)e_{\boldsymbol z},
\qquad \boldsymbol z\in\{0,1\}^N.
\]
Thus, the operator $\exp(-i\Delta t\, s(t_k)\mathcal H_F)$ acts by multiplying each computational basis vector by a phase determined by the objective value $E(\boldsymbol z)$. More generally, for a state vector $\psi \in \mathcal V$, with $\psi=\sum_{\boldsymbol z} c_{\boldsymbol z} e_{\boldsymbol z}$,
the action is
\[
\exp\!\left(-i\,\Delta t\, s(t_k)\mathcal H_F\right)\psi
=
\sum_{\boldsymbol z} c_{\boldsymbol z}
\exp\!\left(-i\,\Delta t\, s(t_k)E(\boldsymbol z)\right)e_{\boldsymbol z}.
\]

This diagonal structure enables a decomposition of $\mathcal H_F$ into elementary diagonal operators.
To demonstrate, recall the Pauli-$Z$ matrix \cite{Nannicini25},
\[
Z = \begin{pmatrix} 1 & 0 \\ 0 & -1 \end{pmatrix},
\]
and define, for each $i=1,\dots,N$, the operator acting on the $i$th factor of $\mathcal V = (\mathbb C^2)^{\otimes N}$ by $Z_i := I^{\otimes(i-1)} \otimes Z \otimes I^{\otimes(N-i)}$. Since the computational basis vectors are indexed by binary configurations, it is convenient to introduce $\hat z_i := \frac{I - Z_i}{2}$. Then $\hat z_i$ is diagonal in the computational basis and satisfies $\hat z_i e_{\boldsymbol z} = z_i e_{\boldsymbol z}$.

Using this notation, the Hamiltonian $\mathcal H_F$ can be written as $\mathcal H_F = \sum_{i,j=1}^N Q_{ij}\hat z_i \hat z_j + \sum_{i=1}^N p_i \hat z_i$. Since all operators $\hat z_i$ and $\hat z_i\hat z_j$ commute, the corresponding exponential operator factorizes exactly, and we obtain
\begin{equation} \label{eq:discrete-final}
\begin{aligned}
\exp\!\left(-i\, \Delta t\, s(t_k)\mathcal H_F\right)\psi
&=
\left(
\prod_{i,j=1}^N
\exp\!\left(
- i\,\Delta t\, s(t_k)\, Q_{ij}\hat z_i \hat z_j
\right)
\prod_{i=1}^N
\exp\!\left(
- i\,\Delta t\, s(t_k)\, p_i \hat z_i
\right)
\right)\psi.
\end{aligned}
\end{equation}
The decomposition above leads to elementary unitary operations that can be implemented on a quantum computer. In particular, operators of the form $\exp(-i\,\theta\, \hat z_i)$ and $\exp(-i\,\theta\, \hat z_i \hat z_j)$ correspond to single-qubit and two-qubit phase operations, respectively, where the parameter $\theta$ is determined by $\Delta t\, s(t_k)\, p_i$ or $\Delta t\, s(t_k)\, Q_{ij}$.

Having defined the final Hamiltonian $\mathcal H_F$, we next introduce an initial Hamiltonian $\mathcal H_I$ whose ground state can be prepared efficiently and whose associated evolution can be implemented using elementary operations. 

A standard choice for the initial Hamiltonian is $\mathcal H_I = - \sum_{i=1}^N X_i$, where $X_i := I^{\otimes(i-1)} \otimes X \otimes I^{\otimes(N-i)}$ and $X=\begin{pmatrix}0&1\\1&0\end{pmatrix}$. The ground state of $\mathcal H_I$ is the uniform superposition
$
\phi_1 = 2^{-N/2} \sum_{\boldsymbol z \in \{0,1\}^N} e_{\boldsymbol z},
$
which can be prepared efficiently. The associated unitary evolution operator factorizes as (cf.~\eqref{eq:totter})
\begin{equation} \label{eq:discrete-initial}
\exp\!\left(-i \Delta t(1 - s(t_k)) \mathcal H_I\right)
=
\prod_{i=1}^N \exp\!\left(i \Delta t(1 - s(t_k)) X_i\right),
\end{equation}
and the evolution is implemented efficiently as a sequence of single-qubit rotations.

To summarize, for a given QUBO problem specified by $(Q,\boldsymbol p)$, we choose a time step $\Delta t$ and a schedule function $s(t)$, and construct the sequence of exponential operators in \eqref{eq:discrete-initial} and \eqref{eq:discrete-final}. Starting from the ground state $\phi_0$ of $\mathcal H_I$, we apply these operators sequentially as in \eqref{eq:quantum-evolution} and \eqref{eq:totter}, over a total evolution time $T$. The effectiveness of this approach depends on adiabatic parameters such as the evolution time, smoothness of the Hamiltonian path, and spectral gap; these are summarized in the next section and determine the likelihood of recovering a QUBO minimizer.

\subsection{Conditions on adiabatic parameters}
\label{sec:adiabatic_params}
In this work, we restrict ourselves to linear interpolation $\mathcal{H}(s) = (1-s)\mathcal{H}_I + s\mathcal{H}_F$ between the Hamiltonians. In the following, we present conditions on adiabatic parameters restricted to linear interpolation. Theorem~\ref{thm:adiabatic} identifies the total evolution time $T$ as
the key quantity governing adiabatic accuracy, but in practice the
continuous evolution must be discretized and the final Hamiltonian
$\mathcal{H}_F$ must be approximated. This introduces two further parameters: the number of Trotter steps $p$ and a sparsification
threshold $\varepsilon$.

We quantify sufficiently large $T$, stated in \ref{thm:adiabatic} 
as follows~\cite[Cor. 9.24]{Nannicini25}: There exists $c>0$ such that
$$T \geq \frac{c}{\varepsilon}\left( \frac{\| \mathcal{H_F}-\mathcal{H_I}\|_{\mathrm{op}}}{\gamma^2}+\frac{\| \mathcal{H_F}-\mathcal{H_I}\|_{\mathrm{op}}^2 }{\gamma^3}\right).$$
This bound is highly conservative and is typically unattainable under realistic experimental conditions.
It is stated in~\cite[Section~9.1.5]{Nannicini25} that for the main statement of \ref{thm:adiabatic} to hold, it is
sufficient to choose
\begin{equation}\label{eq:adiabatic_bound}
  T \;\gg\; \int_0^1 \frac{\|\mathcal{H}'(s)\|_{\mathrm{op}}}{\gamma(s)^2}\,\mathrm{d}s.
\end{equation}
Note that for linear interpolation
$\partial_s \mathcal{H}(s) = \mathcal{H}_F - \mathcal{H}_I$, 
the integration simplifies to
\begin{equation}\label{eq:T_condition}
    T \;\gg\; \frac{\|\mathcal{H}_F - \mathcal{H}_I\|_{\mathrm{op}}}{\gamma^2},
\end{equation}
where $\gamma := \inf_{s \in [0,1]} \gamma(s)$ is the minimum spectral
gap along the interpolation path.
Thus, the evolution time increases with the Hamiltonian variation and decreases with the square of the spectral gap.
Since $\gamma$ is generally intractable,
it must be estimated
numerically, with details discussed explicitly in
\ref{sec:results2}.

In the Trotterized implementation described in
\eqref{eq:quantum-evolution}--\eqref{eq:totter}, the continuous evolution
is replaced by $p$ alternating applications of
$e^{-i\beta_k\mathcal{H}_I}$ and $e^{-i\gamma_k\mathcal{H}_F}$
with
$\beta_k=(1-s_k)\Delta t$ and $\gamma_k=s_k \Delta t$ and step size $\Delta t = T/p$.
By the first-order Lie--Suzuki--Trotter product formula
\cite[Section~6.2]{Nannicini25}, the error of approximating
$e^{-i(\mathcal{H}_1+\mathcal{H}_2)t}$ by $e^{-i\mathcal{H}_1 t}e^{-i\mathcal{H}_2 t}$ for fixed $H_1$,
$H_2$ scales as $\mathcal{O}(\|\mathcal{H}_1\|_{\mathrm{op}}\|\mathcal{H}_2\|_{\mathrm{op}}t^2)$.
At layer $k$, the effective Hamiltonians are $\mathcal{H}_1 = s_k\mathcal{H}_F$
and $\mathcal{H}_2 = (1-s_k)\mathcal{H}_I$ with $s_k \in [0,1]$, giving a
per-step error of
$\mathcal{O}(s_k(1-s_k)\|\mathcal{H}_F\|_{\mathrm{op}}\|\mathcal{H}_I\|_{\mathrm{op}}\Delta t^2)
\leq \mathcal{O}(\|\mathcal{H}_F\|_{\mathrm{op}}\|\mathcal{H}_I\|_{\mathrm{op}}\Delta t^2)$,
where we used $s_k(1-s_k) \leq 1$.
For fixed $T$, increasing $p$ reduces $\Delta t$ and hence the Trotter
error. However, this increases circuit depth and the accumulation of
hardware gate errors.

The QUBO Hamiltonian $\mathcal{H}_F$ 
may contain up to $\binom{n}{2}$ two-body interaction terms, one for
each pair of binary variables; retaining more terms results in a
deeper circuit and a larger accumulation of hardware gate errors.
To reduce circuit depth, interaction terms with coefficient below
a threshold $\varepsilon$ are set to zero; the resulting sparsified
Hamiltonian $\mathcal{H}_F^\varepsilon$ retains only the dominant
interaction terms. In practice, $\varepsilon$ is chosen empirically
to balance retention of problem structure against circuit depth.

The three parameters $T$, $p$, and $\varepsilon$ are interdependent
and generally cannot all be set favourably simultaneously. A large $T$
improves adiabatic fidelity but, for fixed $p$, increases $\Delta t$
and hence the Trotter error. A large $p$ reduces the Trotter error
but increases circuit depth and hardware noise. A small $\varepsilon$
retains more problem structure but may increase circuit depth.
In the near-term hardware setting, practical implementations therefore
operate far from the conditions of Theorem~\ref{thm:adiabatic}.
Despite this, good solutions can still be recovered through post-selection over multiple measurement shots, even when the distribution is not concentrated at the ground state. These trade-offs are quantified in \ref{sec:results2}.

\section{Quantum Variational Regularization (QVaR) Method} \label{sec:qvar}

In this section, we reformulate the regularized variational problem \eqref{eqn:VarProb} as a QUBO problem \eqref{eq:QUBO} and apply the procedure of \ref{sec:QUBO} to compute its minimizer.

To summarize the QVaR method, we first rewrite \ref{eqn:VarProb} as a quadratic objective function over binary variables. We then define the initial and final Hamiltonians, $\mathcal{H}_I$ and $\mathcal{H}_F$, respectively, and construct a sequence of unitary evolution operators (see \ref{eq:quantum-evolution}), using the discretizations in \ref{eq:discrete-final} and \ref{eq:discrete-initial}. This sequence implements an adiabatic evolution that transforms the ground state of $\mathcal{H}_I$ into that of $\mathcal{H}_F$, which solves the original minimization problem.

In this work, we assume that the unknown function $f$ admits a representation through a predefined sparsifying transform. Specifically, we introduce a latent sequence $\boldsymbol{x} \in \mathbb R^d$ and an associated synthesis operator $\Phi$ such that $f \approx \Phi \boldsymbol{x} \in \mathbb R^n$. The operator $\Phi\in \mathbb R^{n\times d}$ can be interpreted as a collection of basis, for instance, arising from wavelet representations.

Let us start with the Tikhonov regularized problem, where $\mathcal R(f) = \| Bf \|^2_2$. Under this parameterization, the variational problem in \ref{eqn:VarProb} can be reformulated as
\begin{equation} \label{eq:variational-latent}
\boldsymbol x^\star = \arg\min_{\boldsymbol x \in \mathbb R^d } \|A \Phi \boldsymbol{x} - g\|_{2}^2 + \alpha \|\boldsymbol{x}\|_{2}^2.
\end{equation}
When the sparsifying transform $B$ is invertible, one may identify $\Phi = B^{-1}$, yielding an equivalent formulation in terms of the latent coefficients $\boldsymbol{x} = B f$.

We now introduce a finite set of candidate solutions $\{\boldsymbol{x}_1,\dots,\boldsymbol{x}_M\}$ together with corresponding binary encodings $\{\boldsymbol{\beta}_1,\dots,\boldsymbol{\beta}_M\}$. Each encoding $\boldsymbol{\beta}_i$ consists of binary variables that uniquely parameterize the candidate solution $\boldsymbol{x}_i$ through a linear encoding map. As an illustrative example, consider a uniform discretization of the unit cube $[0,1]^d$. In this setting, each coordinate of $\boldsymbol{x}_i$ is represented using $k$ binary digits according to $x_{i,\ell}=\sum_{j=1}^{k}2^{-j}\beta_{i,\ell}^{j}$, where $\beta_{i,\ell}^{j}\in\{0,1\}$ and $\ell=1,\dots,d$. This defines a structured mapping between binary encodings and discretized points in $[0,1]^d$.



Let $C$ denote the linear encoding matrix mapping binary representations to candidate solutions, $\boldsymbol{x}_i=C\boldsymbol{\beta}_i$, and define
\[
Q=C^\top\Phi^\top A^\top A\Phi C+\alpha C^\top C,
\qquad
\boldsymbol{p}=-2C^\top\Phi^\top A^\top y,
\qquad
c=g^\top g.
\]
Then, \eqref{eq:variational-latent} is equivalent to the QUBO problem \eqref{eq:QUBO}, with the constant term $c$ having no effect on the minimizer.

Now suppose that the regularization term in \eqref{eqn:VarProb} is given by $\mathcal{R}(f) = \alpha \|B f\|_{1}$. Proceeding as above, and introducing the parametrization $f \approx \Phi \boldsymbol{x}$, we obtain the optimization problem
\begin{equation} \label{eq:variational-l1}
\boldsymbol{x}^\star = \arg\min_{\boldsymbol{x} \in \mathbb{R}^d} \|A \Phi \boldsymbol{x} - g\|_{2}^2 + \alpha \|\boldsymbol{x}\|_{1}.
\end{equation}
However, the resulting optimization problem cannot be directly cast as a QUBO. Indeed, the $1$-norm term introduces absolute values, $\|\boldsymbol{x}\|_{\ell^1} = \sum_i |x_i|$,
which are non-quadratic and non-smooth. To overcome this, we first reformulate the unconstrained minimization problem into a constrained minimization problem with $
\boldsymbol{x} = \boldsymbol{x}^+ - \boldsymbol{x}^-$, $ 
\boldsymbol{x}^+ \geq 0$, and $\boldsymbol{x}^- \geq 0.
$ as
\[
\boldsymbol{x}^{\star,+},\boldsymbol{x}^{\star,-}
=
\arg\min_{\boldsymbol{x}^+,\boldsymbol{x}^-}
\left\|
A\Phi(\boldsymbol{x}^+-\boldsymbol{x}^-)-g
\right\|_{2}^2
+
\alpha \mathbf{1}^\top(\boldsymbol{x}^+ + \boldsymbol{x}^-),
\]
\[
\text{subject to}
\qquad
\boldsymbol{x}^+ \geq 0,
\qquad
\boldsymbol{x}^- \geq 0.
\]
where $\mathbf{1}$ denotes a vector of ones. We now select candidate representations for $\boldsymbol{x}^+$ and $\boldsymbol{x}^-$ that are nonnegative by construction, thereby enforcing the positivity constraints implicitly while preserving a quadratic objective function. To this end, we introduce binary encodings $\boldsymbol{\beta}^+$ and $\boldsymbol{\beta}^-$, analogous to the previous construction, to encode $\boldsymbol{x}^+$ and $\boldsymbol{x}^-$. Defining 
\[
Q = \widetilde C^\top \Phi^\top A^\top A \Phi \widetilde C, \qquad \boldsymbol{q} = -2\widetilde C^\top \Phi^\top A^\top y + \alpha \widehat C^\top \mathbf{1}, \qquad c = y^\top y,
\]
with $
\boldsymbol{z}
=
[
\boldsymbol{\beta}^+,
\boldsymbol{\beta}^-
] ^\top
\in\{0,1\}^{2N}$, $\widetilde C =
[
C^+, -C^-
]$, and
$\widehat C =
[
C^+, C^-
]$,
then \eqref{eq:variational-l1} takes the QUBO formulation in \eqref{eq:QUBO}.


The above reformulation is not unique. An alternative reformulation introduces auxiliary variables to eliminate the absolute value. Let $\boldsymbol{b} = Bf$ and introduce $\boldsymbol{t}, \boldsymbol{s}, \boldsymbol{r} \geq 0$ such that $
\boldsymbol{t} - \boldsymbol{b} - \boldsymbol{s} = 0$, and
$\boldsymbol{t} + \boldsymbol{b} - \boldsymbol{r} = 0$.
This leads to the quadratic penalty formulation
\begin{equation}
\begin{aligned}
\boldsymbol{x}_1^\star,\boldsymbol{x}_2^\star,&\boldsymbol{x}_3^\star,\boldsymbol{x}_4^\star,\boldsymbol{x}_5^\star
=\\
&\arg\min_{\boldsymbol{x}_1,\boldsymbol{x}_2,\boldsymbol{x}_3,\boldsymbol{x}_4,\boldsymbol{x}_5}
\;
\left\|
A\Phi \boldsymbol{x}_1 - g
\right\|_{2}^2
+
\alpha \mathbf{1}^\top \boldsymbol{x}_3
+
\mu \|\boldsymbol{x}_2 - B\Phi \boldsymbol{x}_1\|_{2}^2
 \\
& +
\nu \left(
\|\boldsymbol{x}_3 - \boldsymbol{x}_2 - \boldsymbol{x}_4\|_{2}^2
+
\|\boldsymbol{x}_3 + \boldsymbol{x}_2 - \boldsymbol{x}_5\|_{2}^2
\right), \qquad \nu>0.
\end{aligned}
\end{equation}
Therefore, there are multiple possible ways to recast an $\ell^1$-regularized inverse problem as a quadratic optimization problem. The auxiliary-variable formulation above avoids the explicit absolute value by introducing the variables $\boldsymbol{b}$, $\boldsymbol{t}$, $\boldsymbol{s}$, and $\boldsymbol{r}$ together with quadratic penalty terms. However, each additional continuous variable must be represented by binary variables in the QUBO encoding, and therefore requires additional qubits in the corresponding QVaR formulation. Although such reformulations are mathematically valid, they may be costly from a quantum-resource perspective. For this reason, we adopt the previous formulation, which introduces fewer auxiliary variables and therefore leads to a more compact binary representation.

The reformulations presented here are classical in optimization and inverse problems, where variational formulations are routinely expressed through data fidelity and regularization terms. The novelty in the present context lies not in the reformulation itself, but in its role as a bridge to quantum optimization. While modern regularization strategies often require specialized algorithms and substantial analytical effort, many such problems can be systematically transformed into quadratic optimization problems. This enables a unified treatment within QUBO-based frameworks and provides a systematic pathway for translating inverse problems into quantum-compatible optimization models.

\subsection{Quantum Sensitivity and Bayesian Uncertainty}\label{sec:Uncertainty}

In this section, we relate the distribution of solutions arising from early termination of the quantum algorithm to the uncertainty inherent in linear inverse problems with Tikhonov regularization.

One way to characterize the ill-posedness in inverse problems is through the Bayesian framework. In this setting, the inverse problem \eqref{eqn:InvProb} is reformulated as a statistical model $g = Af+\delta$ where $g,f$ and $\delta$ are now random variables.
This formulation implicitly encodes the regularization parameter $\alpha$ through the noise distribution; for instance, in the case of Tikhonov regularization, one assumes $\delta \sim \mathcal{N}(0, \tfrac{1}{\alpha} I)$ and $g$ denotes the measurement random variable.

Within this framework, the solution to the inverse problem is given by the conditional distribution $f \mid g $, typically characterized by its density $\pi_{f \mid g}$. The variational formulation \eqref{eqn:VarProb} then appears as the negative log-density (up to an additive constant) of this posterior distribution:
\[
\pi_{f \mid g} (f) \propto \exp\left( -\frac{1}{2\alpha}\| Af - g \|^2_{2} - \frac{1}{2} \| Bf \|^2_{2} \right).
\]
Given the change of variable introduced in the previous section, we obtain the posterior density
\[
\pi_{\boldsymbol{x} \mid g}(\boldsymbol{x} )
\;\propto\;
\exp\!\left(
- \frac{1}{2\alpha} \|A \Phi \boldsymbol{x} - g\|_{2}^2
- \frac{1}{2}\|\boldsymbol{x}\|_{2}^2
\right),
\]
where $\boldsymbol{x}$ is now the solution random variable. A standard way to quantify uncertainty in this setting is through the covariance structure of the posterior distribution. By completing the square in the exponent, we obtain
\begin{equation}
\pi_{\boldsymbol{x} \mid g}(\boldsymbol{x} )
\;\propto\;
\exp\!\left(
-\frac{1}{2} (\boldsymbol{x} - \boldsymbol{x}_{\text{MAP}})^\top
C_{\text{post}}^{-1}
(\boldsymbol{x} - \boldsymbol{x}_{\text{MAP}})
\right),
\end{equation}
where $
\boldsymbol{x}_{\text{MAP}}
=
C_{\text{post}} \, \Phi^\top A^\top y
$, and 
$C_{\text{post}}
=
\left(
\Phi^\top A^\top A \Phi + \alpha I
\right)^{-1}$.

Assuming that $\Phi^\top A^\top A \Phi + \alpha I$ is strictly positive definite, the posterior covariance matrix $C_{\text{post}}$ is symmetric positive definite and admits an eigendecomposition with eigenvalues $\lambda_1 \geq \lambda_2 \geq \cdots > \lambda_d > 0$ and corresponding eigenvectors $\{\boldsymbol{v}_i\}_{i=1}^d$. These eigenvalues characterize the reconstruction uncertainty: larger eigenvalues correspond to directions in parameter space with higher uncertainty, while smaller eigenvalues indicate directions that are more strongly constrained by the data. Consequently, the spectrum of $C_{\text{post}}$ describes how uncertainty is distributed across different modes and is closely linked to the stability and ill-posedness of the underlying inverse problem.

We now show that these eigendirections also influence the distribution of solutions produced by QUBO-based quantum optimization when the Hamiltonian evolution is terminated early. Note that the QUBO objective $E$ in \eqref{eq:QUBO} coincides with the negative log-posterior $N(\boldsymbol{x}) := -\log \pi_{\boldsymbol{x} \mid g}(\boldsymbol{x})$ up to an additive constant independent of $\boldsymbol{x}$.

Following \eqref{eq:final-hamiltonian}, the action of the final Hamiltonian on a computational basis state $\boldsymbol{e}_{\boldsymbol{z}}$ is $
\mathcal{H}_F \boldsymbol{e}_{\boldsymbol{z}}
=
E(\boldsymbol{z}) \boldsymbol{e}_{\boldsymbol{z}}
=
N(\boldsymbol{x}(\boldsymbol{z})) \boldsymbol{e}_{\boldsymbol{z}},
$
where constants independent of $\boldsymbol{z}$ are omitted. If the posterior mode is representable in the chosen parametrization, i.e., if there exists $\boldsymbol{z}^\star\in\{0,1\}^N$ such that $\boldsymbol{x}_{\mathrm{MAP}}=\boldsymbol{x}(\boldsymbol{z}^\star)$, then $\boldsymbol{z}^\star$ minimizes \eqref{eq:QUBO}, and the corresponding reconstruction equals $\boldsymbol{x}_{\mathrm{MAP}}$.

To understand how the energy behaves near the ground state, consider a perturbation along an eigenvector $\boldsymbol{v}_i$ of $C_{\mathrm{post}}$, i.e., $\boldsymbol{x} := \boldsymbol{x}_{\mathrm{MAP}} + s \boldsymbol{v}_i$. Using the parametrization $\boldsymbol{x} = \Phi \boldsymbol{z}$ and assuming that $\Phi$ is invertible, this corresponds to $\boldsymbol{z} := \boldsymbol{z}^\star + s \, \Phi^{-1} \boldsymbol{v}_i$, where $\boldsymbol{x}_{\mathrm{MAP}} = \Phi \boldsymbol{z}^\star$.

The corresponding energy increase is given by
\begin{equation} \label{eq:energy-excess}
\Delta E := E(\boldsymbol{x}) - E(\boldsymbol{x}_{\mathrm{MAP}})
= \frac{1}{2} (s \boldsymbol{v}_i)^T C_{\mathrm{post}}^{-1} (s \boldsymbol{v}_i)
= \frac{s^2}{2} \boldsymbol{v}_i^T C_{\mathrm{post}}^{-1} \boldsymbol{v}_i
= \frac{s^2}{2 \lambda_i},
\end{equation}
where we used that $C_{\mathrm{post}} \boldsymbol{v}_i = \lambda_i \boldsymbol{v}_i$ and that the eigenvectors are normalized.

Therefore, in directions corresponding to larger eigenvalues $\lambda_i$ (i.e., higher posterior uncertainty), the energy increases more slowly, leading to smaller local energy gaps. Conversely, directions with smaller eigenvalues (lower uncertainty) exhibit a steeper increase in energy and thus larger energy gaps.

We now connect this observation to the adiabatic theorem under finite-time evolution. At time $T$, the quantum state can be written as $\psi(T)=\sum_{\boldsymbol{z}} c_{\boldsymbol{z}}(T)\boldsymbol{e}_{\boldsymbol{z}}$, where $\sum_{\boldsymbol{z}} |c_{\boldsymbol{z}}(T)|^2 = 1$. By the Born rule \cite{sakurai1986modern}, the probability of measuring $\boldsymbol{e}_{\boldsymbol{z}}$ is $q_T(\boldsymbol{z}) := |\langle \boldsymbol{e}_{\boldsymbol{z}}, \psi(T) \rangle|^2 = |c_{\boldsymbol{z}}(T)|^2$. If $\boldsymbol{z}^\star$ denotes the minimizer, then the probability of not observing it is $p_T(\boldsymbol{z}^\star) := 1 - q_T(\boldsymbol{z}^\star) = 1 - |c_{\boldsymbol{z}^\star}(T)|^2$.

Let $\boldsymbol{z}_T$ denote the random variable on $\{0,1\}^N$ with probability law $q_T$, and define the corresponding reconstruction $\boldsymbol{x}_T := C \boldsymbol{z}_T$, with $C$ the binary map matrix. For a normalized eigenvector $\boldsymbol{v}_i$ of $C_{\mathrm{post}}$, define the scalar random variable $S_T^i := \boldsymbol{v}_i^T (X_T - \boldsymbol{x}_{\mathrm{MAP}})$, which measures the displacement along the eigendirection $\boldsymbol{v}_i$. Using \eqref{eq:energy-excess}, the corresponding energy excess is $\Delta E_i = \frac{1}{2} (S_T^i \boldsymbol{v}_i)^T C_{\mathrm{post}}^{-1} (S_T^i \boldsymbol{v}_i) = (S_T^i)^2/2\lambda_i$.

Taking expectations, we obtain $\mathbb{E}[(S_T^i)^2] = 2 \lambda_i \mathbb{E}(\Delta E_i)$, and therefore 
\begin{equation} \label{eq:var-Si}
\mathrm{Var}(S_T^i) = 2 \lambda_i \mathbb{E}(\Delta E_i) - (\mathbb{E}(S_T^i))^2 \leq 2 \lambda_i \mathbb{E}(\Delta E_i),
\end{equation}
with equality achieved when $\mathbb{E}(S_T^i) = 0$ in an ideal case.

Now, let us evaluate $\mathbb E(\Delta E_i)$. Since the state at time $T$ obeys the law $q_T$, we can write this expectation as
\begin{equation}
    \mathbb E(\Delta E) = \sum_{\boldsymbol{z}} q_T(\boldsymbol{z}) ( E(\boldsymbol{x(\boldsymbol{z})}) - E(\boldsymbol{\boldsymbol{x}_{\text{MAP}}}) )=\sum_{\boldsymbol{z}\not = \boldsymbol{z}^\star} q_T(\boldsymbol{z}) ( E(\boldsymbol{x(\boldsymbol{z})}) - E(\boldsymbol{\boldsymbol{x}_{\text{MAP}}}) ).
\end{equation}
Recall $\gamma$ being the first energy gap, and assume that there is a maximum total energy $E_{\text{max}}$ such that we can bound the expected energy gap as
\begin{equation} \label{eq:energy-bounds}
      \mathbb E(\Delta E_i) \leq \mathbb E(\Delta E) \leq E_{\text{max}} \sum_{\boldsymbol{z}\not = \boldsymbol{z}^\star} q_T(\boldsymbol{z}).
\end{equation}
However, $\sum_{\boldsymbol{z}\not = \boldsymbol{z}^\star} q_T(\boldsymbol{z}) = p_T(\boldsymbol{z}^\star)$, and from \ref{thm:adiabatic} we can bound this by
\begin{equation} \label{eq:adiabatic-simple}
    p_T(\boldsymbol{z}^\star) \leq \frac{C}{T^2\gamma^4},
\end{equation}
where the smoothness assumption of Hamiltonian interpolation is absorbed in the constant $C$. Plugging \eqref{eq:adiabatic-simple} and \eqref{eq:energy-bounds} into \eqref{eq:var-Si} we obtain the bounds for the directional distribution of the state along eigen directions of the covariance
\begin{equation} \label{eq:QA-variance}
    \mathrm{Var}(S_T^i) \leq C' \frac{ \lambda_i }{T^2 \gamma^4},
\end{equation}
with $C' = 2 E_{\text{max}} C$. Therefore, the variances along covariance eigendirection $\boldsymbol{v}_i$ is bounded by the Bayesian posterior variance $\lambda_i$, proportional to the adiabatic error factor depending on runtime and the minimum spectral gap. We summarize the discussion above into the theorem below.

\begin{theorem}[Directional and total variance bounds for early-terminated quantum annealing]
Let \(C_{\mathrm{post}}\) be the posterior covariance with eigendecomposition \(C_{\mathrm{post}} v_i = \lambda_i v_i\), where \(\{v_i\}_{i=1}^d\) is an orthonormal basis and \(\lambda_i>0\). Let \(\boldsymbol z_T\) be the random binary variable obtained by measuring the early-terminated state \(\psi(T)\), and define \(\boldsymbol{x}_T=C \boldsymbol z_T\). For each direction, define \(S_T^i := v_i^\top (\boldsymbol{x}_T - x_{\mathrm{MAP}})\). Assume the adiabatic failure probability satisfies \(p_T \le C/(T^2 \gamma ^4)\), and that the excess energy is bounded by \(E_{\max}\). Then the directional variances satisfy \eqref{eq:QA-variance}.
Moreover, the total variance satisfies
\begin{equation}
\operatorname{Tr}(\operatorname{Cov}(X_T))
\le
C'
\frac{\operatorname{Tr}(C_{\mathrm{post}})}{T^2 \gamma^4}.
\end{equation}
\end{theorem}

\begin{proof}
The final inequality is obtained by the fact that $\operatorname{Tr}(C_{\text{post}}) = \sum_i \lambda_i$.
\end{proof}

This result suggests that the geometry of quantum samples is aligned with the posterior geometry. However, we need 
another assumption to exactly extract Bayesian uncertainty from quantum samples. 
One way to approximate the distribution $q_T(\boldsymbol{z})$ is
\begin{equation}
    q_T(\boldsymbol{z})
    \;\approx\;
    \frac{1}{\mathcal Z_T}
    \exp\!\left( - \kappa E(\boldsymbol{z})  \right),
\end{equation}
where $\mathcal Z_T$ is a normalization constant and $\kappa>0$ is an effective inverse temperature that depends on the annealing dynamics. This form corresponds to a Gibbs (Boltzmann) distribution. Such distributions arise naturally in simulated quantum annealing methods based on Monte Carlo techniques, where sampling from a Boltzmann distribution of the problem Hamiltonian is explicit \cite{crosson2021rapid}. Moreover, empirical studies of quantum annealing devices suggest that the output distribution can often be well-approximated by a Gibbs distribution at an effective temperature \cite{nelson2022high}.

Since $E$ is quadratic in the continuous variable $\boldsymbol{x}(\boldsymbol{z})$ (see the arguments above), the induced distribution of $\boldsymbol x_T = C \boldsymbol{z}_T$ is approximately Gaussian in a neighborhood of the MAP point, with covariance
\[
    \operatorname{Cov}(\boldsymbol{x}_T) \approx \kappa^{-1} C_{\text{post}}.
\]
Consequently, the covariance structure of QA samples reflects the posterior covariance up to a scalar temperature factor, allowing estimation of Bayesian uncertainty through calibration of $\kappa$. We close this section by summarizing this into a theorem.
\begin{theorem}[Quantum covariance under Gibbs approximation]
Assume that the quantum output distribution satisfies $q_T(\boldsymbol z)
\propto
\exp(-\kappa E(\boldsymbol z)),
$
for some inverse temperature $\kappa > 0$, and that $E(\boldsymbol x)$ is quadratic in $\boldsymbol x$. Then the induced distribution of $\boldsymbol{x}_T = C\boldsymbol{z}_T$ is approximately Gaussian near the MAP point, with covariance
\begin{equation}
\operatorname{Cov}(\boldsymbol{x}_T)
\approx
\kappa^{-1} C_{\mathrm{post}}.
\end{equation}
\end{theorem}

\section{Numerical experiments} \label{sec:results}

We evaluate the proposed QVaR formulation on small-scale inverse problems that can be encoded as QUBO instances and implemented within the adiabatic framework described above. The experiments are intended to demonstrate the full pipeline: discretization of the variational problem, construction of the corresponding QUBO objective, implementation of the associated quantum evolution, and reconstruction from measured bitstrings.

We consider the following test cases. The first is a scalar linear inverse problem with Tikhonov regularization, used as a minimal example that allows explicit inspection of the QUBO construction before moving to higher-dimensional settings. The second extends this setting to a low-dimensional linear inverse problem. The third considers a 1D deconvolution problem with $\ell^1$-regularization in a Haar wavelet representation, illustrating the variable-splitting formulation introduced in the QVaR section. Results are presented for simulated and experimental quantum hardware.


In the experiments below, we first reformulate each regularized variational problem as a QUBO problem following the constructions of \ref{sec:qvar}. We then apply the QVaR method in \ref{sec:qvar} to generate quantum samples from the resulting QUBO model. As discussed in \ref{sec:Uncertainty}, finite-time evolution and approximate adiabatic parameters induce a distribution over low-energy configurations, whose variability reflects the quantum sensitivity of the underlying inverse problem. We report both the mean reconstruction and, with a slight abuse of terminology, the MAP reconstruction, defined as the sampled configuration with lowest QUBO energy $E(\boldsymbol z)$.

Simulations are carried out using Qiskit, an open-source Python quantum computing framework developed by IBM \cite{javadi2024quantum}. The simulations are executed on the LUMI supercomputer \cite{lumi-eurohpc}, which allows  simulation of quantum circuits with up to 44 qubits. Hardware experiments are performed through a 53-qubit quantum device (Q50) accessed via the LUMI infrastructure. The system is co-developed by VTT Technical Research Centre of Finland and IQM Quantum Computers. All code and data required to reproduce the numerical experiments are available in a public repository \cite{Rautio2026QVAR}.

\subsection{1-scalar linear inverse problem with Tikhonov regularization} \label{sec:results1}
We begin with a simple scalar inverse problem that serves as a minimal test case for the proposed framework. Although elementary, it contains the essential ingredients of a regularized inverse problem and illustrates how the problem can be formulated and solved using the QVaR approach. Let us consider a scalar inverse problem in which the goal is to recover an unknown scalar $f\in \mathbb R$ from data $g\in \mathbb R$ satisfying
$ Af = g$,
where $A \in \mathbb R$ is known. We use Tikhonov regularization and minimize
\begin{equation}\label{eq:1dinv}
\arg\min_{f \in \mathbb R} (Af - g)^2 +\alpha f^2.
\end{equation}

To obtain a binary optimization problem, we restrict $f$ to the interval
$[L,U]=[0.6,2.0]$ and discretize it using $M=2^n$ uniformly spaced grid points,
\[
f_i = L + i\,\Delta x,
\qquad
\Delta x = \frac{U-L}{M-1},
\qquad
i=0,\dots,M-1.
\]
Each grid point is represented by a binary vector
$\boldsymbol{z}_i=(z_i^1,\dots,z_i^n)\in\{0,1\}^n$ satisfying
$
i=\sum_{j=1}^{n} 2^{j-1}z_i^j.
$
Consequently, each binary vector defines a candidate reconstruction
$
f(\boldsymbol{z})
=
L
+
\Delta x
\sum_{j=1}^{n}2^{j-1}z^j.
$
Substitution into the Tikhonov functional yields
\[
E(\boldsymbol{z})
=
\bigl(A f(\boldsymbol{z})-g\bigr)^2
+
\alpha\, f(\boldsymbol{z})^2.
\] 
The corresponding QUBO coefficients are obtained by expanding this expression, following the general construction described in \ref{sec:qvar}.

\begin{figure}[t]
    \begin{subfigure}[b]{0.2\textwidth}
        \centering
        \includegraphics[width=\textwidth]{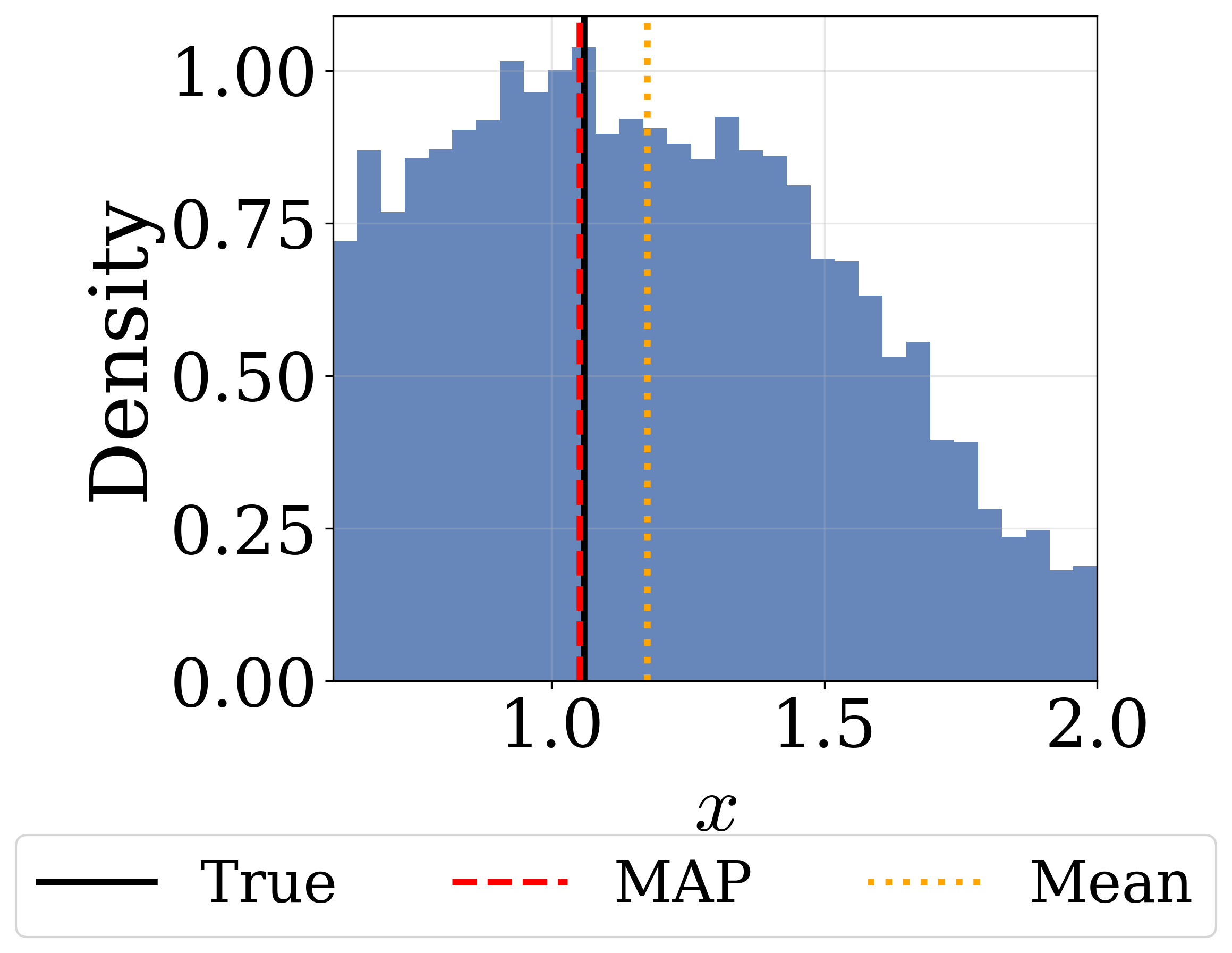}
        \caption{$T=p=2$}
    \end{subfigure}
    \begin{subfigure}[b]{0.2\textwidth}
        \centering
        \includegraphics[width=\textwidth]{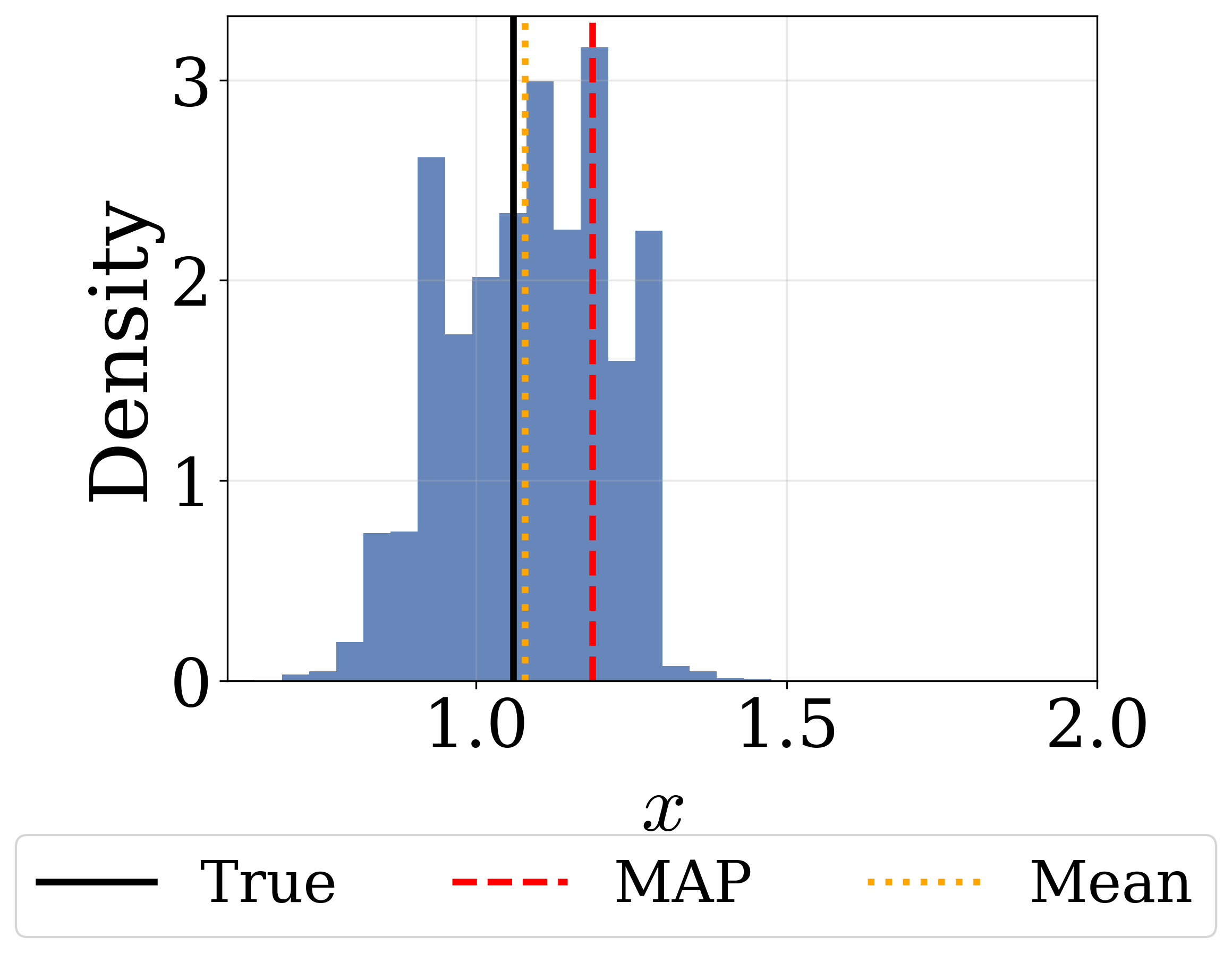}
        \caption{$T=p=100$}
    \end{subfigure}
    \begin{subfigure}[b]{0.2\textwidth}
        \centering
        \includegraphics[width=\textwidth]{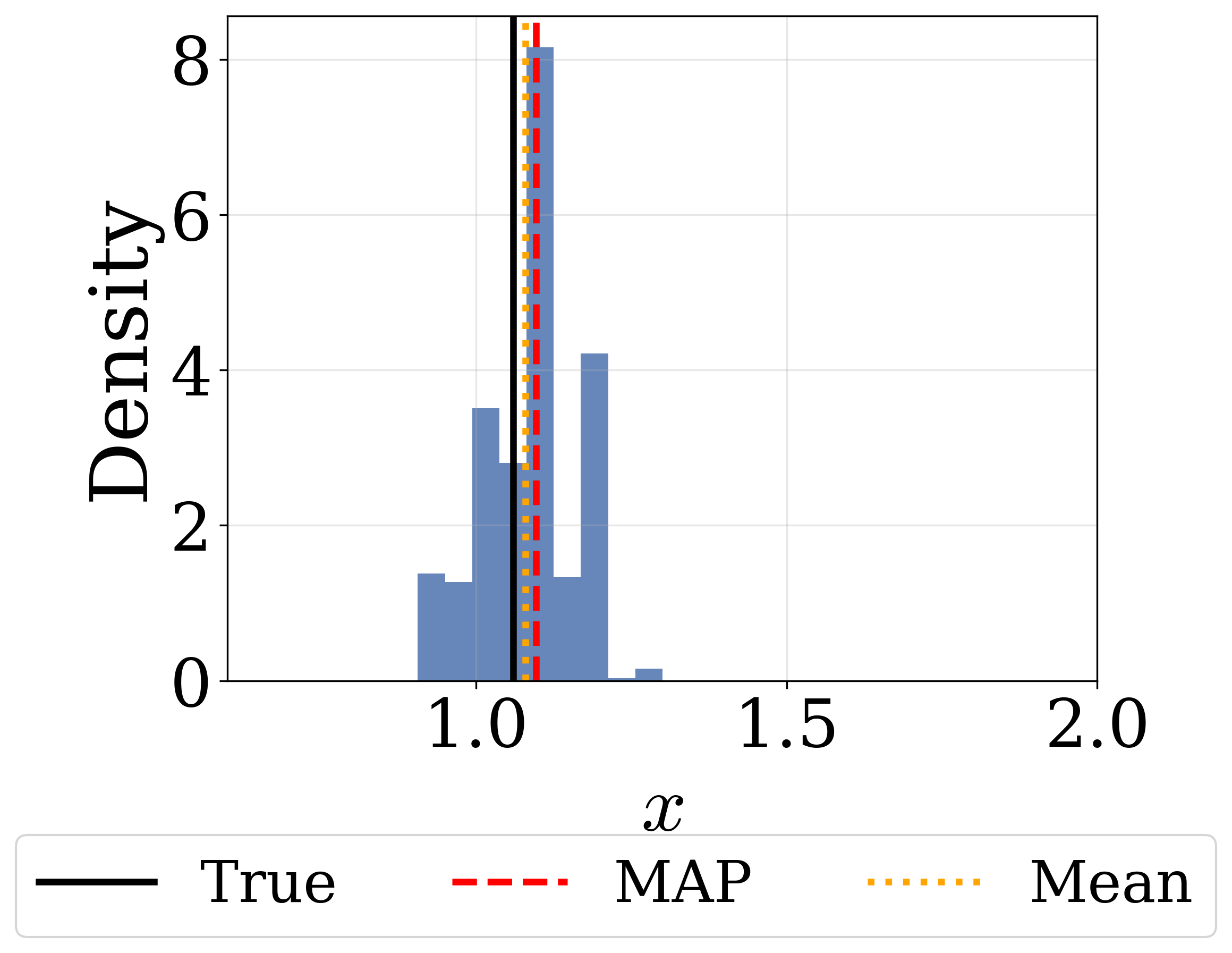}
        \caption{$T=p=1000$}
    \end{subfigure}
    \begin{subfigure}[b]{0.3\textwidth}
        \centering
        $x_{\text{true}} = 1.05962$
        \vspace{0.1cm} \\
        
        \begin{tabular}{c c c}
            \hline
            $T = p$ & $x_{\text{mean}}$ & $x_{\text{MAP}}$ \\
            \hline
            2 & 1.17539 & \textbf{1.05161} \\
            100 & \textbf{1.07898} & 1.18710 \\
            1000 & 1.07953 & 1.09677\\
            \hline
            
        \end{tabular}
        \caption{Numerical results}
    \end{subfigure}
    
    \caption{Simulations for the one-dimensional inverse problem in \eqref{eq:1dinv} when comparing parameter values time $T$ and time discretization steps $p$. As $T$ and $p$ increase, the solution concentrates more sharply around the true value $x_{\text{true}} = 1.05962$.
    }
    \label{fig:experiment1_simulated}
\end{figure}

\begin{figure}[t]
    \begin{subfigure}{0.2\textwidth}
        \centering
        \includegraphics[width=\textwidth]{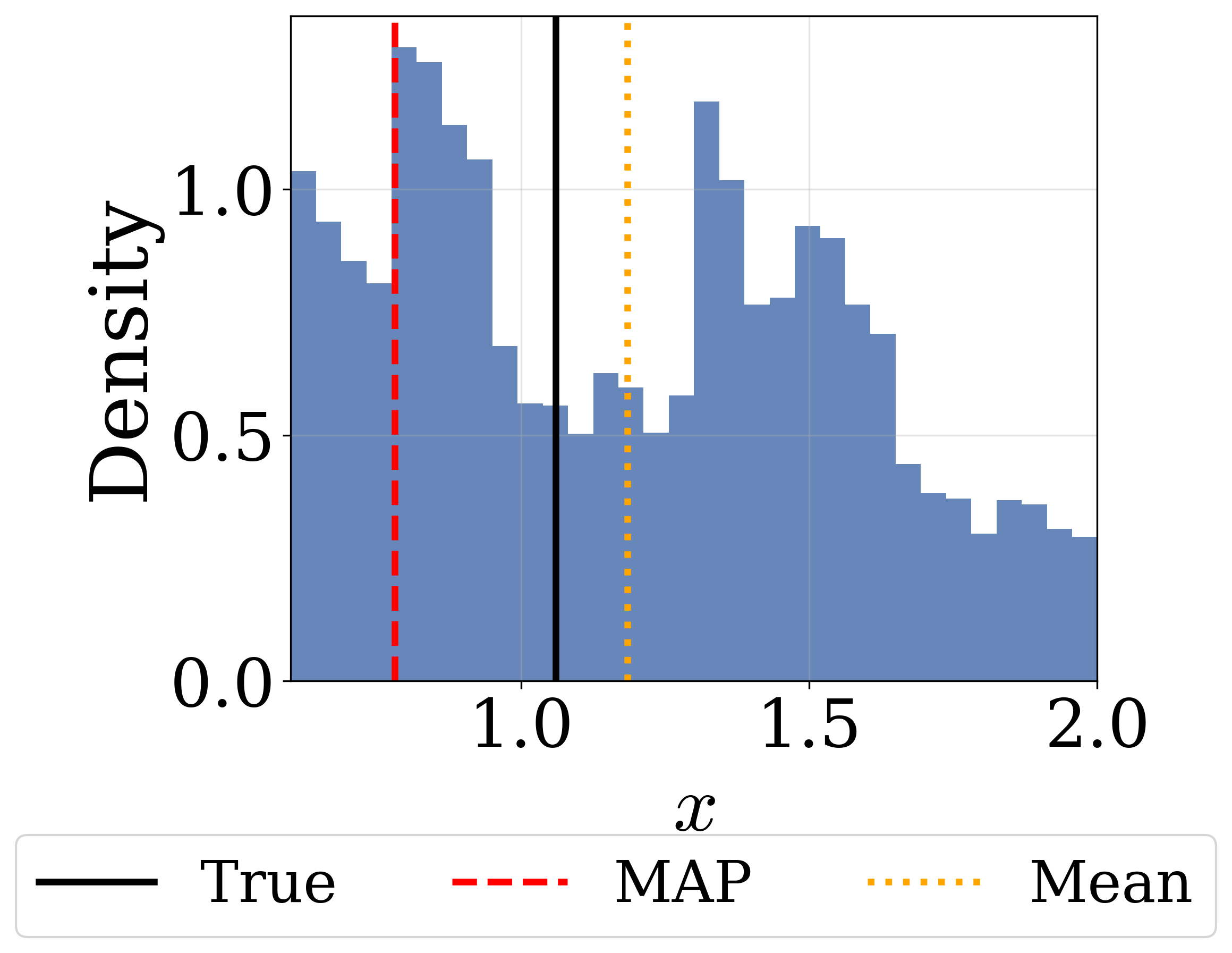}
        \caption{$T=p=2$}
    \end{subfigure}
    \begin{subfigure}{0.2\textwidth}
        \centering
        \includegraphics[width=\textwidth]{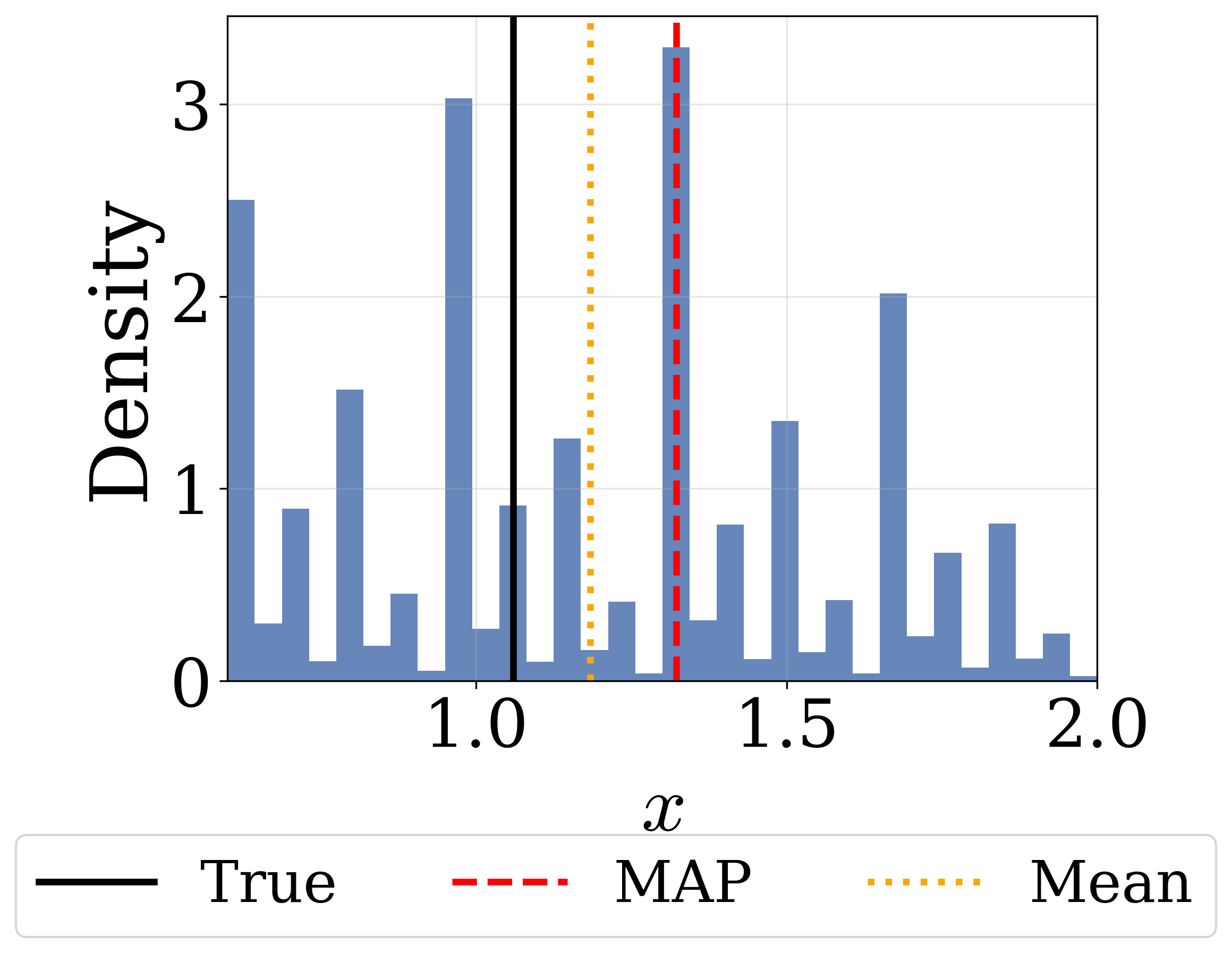}
        \caption{$T=p=100$}
    \end{subfigure}
    \begin{subfigure}{0.2\textwidth}
        \centering
        \includegraphics[width=\textwidth]{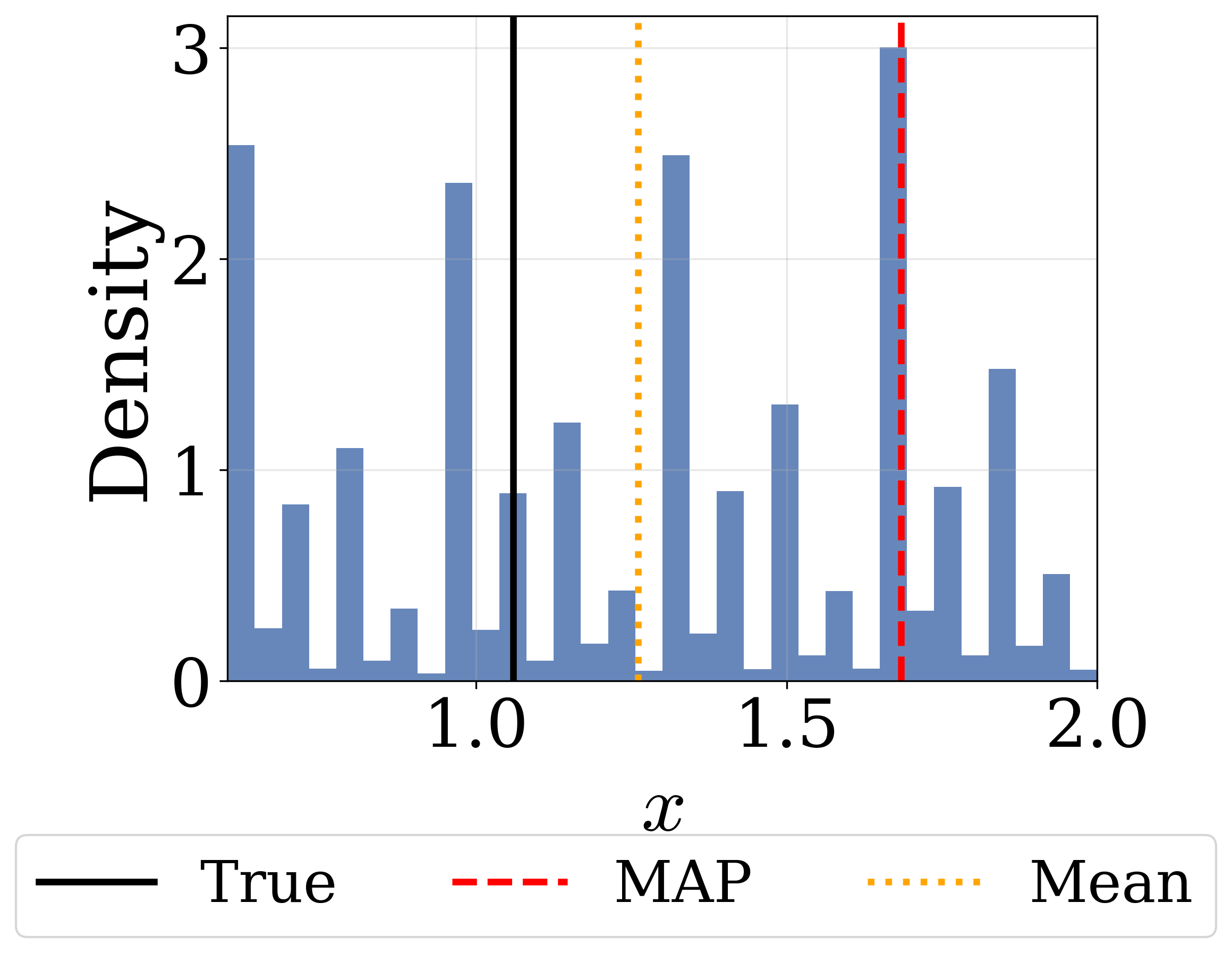}
        \caption{$T=p=200$}
    \end{subfigure}
    \begin{subfigure}{0.28\textwidth}
    \centering
    $x_{\text{true}} = 1.05962$
    \vspace{0.1cm} \\
    \begin{tabular}{c c c}
        \hline
        $T = p$ & $x_{\text{mean}}$ & $x_{\text{MAP}}$ \\
        \hline
        2   & \textbf{1.18413} & 0.78065 \\
        100 & 1.18420          & \textbf{1.32258} \\
        200 & 1.26090          & 1.68387 \\
        \hline
    \end{tabular}
    \caption{Numerical results}
\end{subfigure}
    
    \caption{Results for the one-dimensional inverse problem in \eqref{eq:1dinv} on the VTT Q50 quantum processor. Making the circuit deeper accumulates quantum errors, motivating the choice of smaller $T$ and $p$ values for further experiments.}
    \label{fig:exp1_q50}
\end{figure}


In this numerical experiment, we set $A = 1.7$ and $g = 2.3$, and choose the regularization parameter as $\alpha = 0.8$. Although the analysis in \ref{sec:adiabatic_params} is based on linear interpolation of Hamiltonians, here we use a smooth sine annealing schedule. This choice provided a more consistent performance for the results in this work. Furthermore, we truncate couplings with magnitude below $\epsilon=0.04$ in the simulated case, and $\epsilon=0.15$ for the hardware case. The circuits are executed for 10,000 shots, producing quantum samples over the $2^5$ possible bitstrings. 

We experiment with different choices of the total evolution time $T$ and the number of Trotter steps $p$ used to approximate the continuous annealing evolution over $[0,T]$, for both quantum hardware and noiseless simulation cases. These parameters control the accuracy of the adiabatic approximation: increasing $T$ improves the fidelity of the evolution toward the ground state, while increasing $p$ refines the approximation of the continuous-time dynamics. Insufficient values of either parameter lead to deviations from the ideal evolution.



On the noiseless simulator, increasing $T$ and $p$ sharpens the distribution around the minimizer, as expected from the annealing interpretation (cf. \ref{fig:experiment1_simulated}). However, the Q50 hardware results show the opposite trend for large $T$ and $p$, as presented in \ref{fig:exp1_q50}.
Here, the shallow circuit with small $T$ and $p$ give the most useful estimates, despite being further from the ideal adiabatic limit. In particular, for $T=p=2$, the sample mean remains close to the true solution, even though the MAP estimate is less accurate.  
Deeper circuits accumulate gate and measurement errors, causing the sampled distribution to be dominated by noise.
This motivates the use of small $T,p$ values in the following larger scale experiments, where circuit depth becomes a limiting factor.

\subsection{$3\times2$ linear inverse problem with Tikhonov regularization} \label{sec:results2}
We next consider a low-dimensional linear inverse problem, inspired by examples in \cite{kaipio2005statistical}, with Tikhonov regularization,
\begin{equation}\label{eq:scalar3x2}
    \begin{aligned}
        &\min_f \|Af - g\|^2_2 + \alpha \|x\|^2_2,\\
\text{with}\hspace{6mm}&A=\begin{bmatrix} 1 & -1\\ 1 & -2\\ 2 & 1 \end{bmatrix}, \quad 
g=\begin{bmatrix} 2.1\\ 2.9\\ 1.1 \end{bmatrix}, \quad 
\alpha = 0.5.
    \end{aligned}
\end{equation}
The unknown parameter is $f=(f_1,f_2)\in [0.5,1.5]\times[-1.5,-0.5]$. The corresponding Tikhonov minimizer is
$f^\star \approx (0.9697,-0.8970)$.


Following the construction in \ref{sec:qvar} and choosing $\Phi=I$, the identity matrix, each component of $f$ is discretized using $n=6$ bits per variable, resulting in a total of $2n=12$ qubits. This yields $2^n=64$ candidate values per component and a total of $2^{2n}=4096$ discretized points in the domain. Each bitstring therefore encodes a candidate reconstruction $f(\boldsymbol z)$, and the associated QUBO objective corresponds to the discretized Tikhonov functional.

In contrast to the scalar case, the energy landscape is now defined over a two-dimensional grid, leading to a larger and more structured QUBO problem. This allows us to assess how the sampling-based approach scales with increasing dimensionality while remaining within a regime that is still tractable for both simulation and hardware.



We use $p=4$ time-discretization layers with total evolution time $T=2$,
step size $\Delta t = T/p = 0.5$, and sparsification threshold
$\varepsilon=0.05$.
After sparsification, $32$ out of $66$ possible interaction pairs are
retained.
As discussed in \ref{sec:adiabatic_params}, three sources of error
affect the output distribution: violation of the adiabatic condition,
Trotter discretization error, and hardware gate noise.

We compute the spectral gap $\gamma(s)=E_2(s)-E_1(s)$ of the interpolating Hamiltonian $\mathcal H(s)=(1-s)\mathcal H_I+s\mathcal H_F$ by evaluating $51$ uniformly spaced values of $s\in[0,1]$ and computing the two smallest eigenvalues of $\mathcal H(s)$ in the full $2^{12}=4096$-dimensional Hilbert space using a sparse Lanczos eigensolver \cite{GoVa13}; the result is shown in \ref{fig:spectral_gap}. The gap decreases monotonically from the exact value $\gamma(0)=2$ to $\Delta_{\min}:=\min_{s\in[0,1]}\gamma(s)=0.001020$ at $s=1$, where the two lowest-energy states of $\mathcal H_F$ are nearly degenerate. The value $\gamma(0)=2$ follows from the spectrum of $\mathcal H_I=-\sum_i X_i$, whose two lowest eigenvalues are $-n$ and $-(n-2)$ for $n$ qubits.

Our implementation uses the smooth schedule $s(\tau) = \sin^2(\pi\tau/2)$
with $\tau \in [0,1]$, which is non-linear in $\tau$. The bound
\eqref{eq:T_condition} from \ref{sec:adiabatic_params} was stated
for linear interpolation, where $\partial_s \mathcal{H}(s) = \mathcal{H}_F - \mathcal{H}_I$;
for our non-linear schedule one instead has
$\sup_{\tau \in [0,1]} \|\partial_\tau \mathcal{H}(\tau)\|_{\mathrm{op}}
= (\pi/2)\|\mathcal{H}_F - \mathcal{H}_I\|_{\mathrm{op}}$, where
$\pi/2$ is the maximum of $|\dot s(\tau)| = (\pi/2)\sin(\pi\tau)$.
Applying \eqref{eq:T_condition} up to this order-one constant factor,
we use $\|\mathcal{H}_F - \mathcal{H}_I\|_{\mathrm{op}} \leq
\|\mathcal{H}_F\|_{\mathrm{op}} + \|\mathcal{H}_I\|_{\mathrm{op}}$.
Since $\mathcal{H}_F$ is diagonal, its spectral norm equals
$\max_k |(\mathcal{H}_F)_{kk}| = 1.9234$.
Since $\mathcal{H}_I = -\sum_i X_i$, its spectral norm is
$\|\mathcal{H}_I\|_{\mathrm{op}} = n = 12$ exactly.
By the triangle inequality:
\begin{equation}\label{eq:dH_bound}
  \|\mathcal{H}_F - \mathcal{H}_I\|_{\mathrm{op}} \leq
  \|\mathcal{H}_F\|_{\mathrm{op}} + \|\mathcal{H}_I\|_{\mathrm{op}}
  \approx 1.9234 + 12 = 13.9234.
\end{equation}
Approximating $\gamma(s) \geq \Delta_{\min}$ throughout gives
\begin{equation}\label{eq:T_adiabatic}
  T_{\mathrm{adiabatic}}
  \approx \frac{13.9234}{(0.001020)^2}
  \approx 1.34 \times 10^7.
\end{equation}
With $T = 2$ we have $T / T_{\mathrm{adiabatic}} \approx 1.5\times
10^{-7}$, so the algorithm operates far outside the adiabatic regime
and concentration of probability at the ground state is not guaranteed.
We note that the bound~\eqref{eq:T_adiabatic} is conservative: the
triangle inequality in~\eqref{eq:dH_bound} can be loose, and
approximating $\gamma(s) \approx \Delta_{\min}$ ignores the larger
gap values at $s < 1$ (see \ref{fig:spectral_gap}, right panel).
Nevertheless, since $\Delta_{\min}$ is so small, even a tighter bound
would still yield $T_{\mathrm{adiabatic}} \gg 2$.

Following \ref{sec:adiabatic_params}, the per-layer Trotter error is
$\mathcal{O}(\|\mathcal{H}_F\|_{\mathrm{op}}\|\mathcal{H}_I\|_{\mathrm{op}}\Delta t^2)
= \mathcal{O}(1.9234 \cdot 12 \cdot 0.25)
= \mathcal{O}(5.77)$.
Since operator-norm distances between unitaries are bounded by $2$,
a per-layer bound of $\mathcal{O}(5.77)$ is vacuous and provides no
quantitative guarantee on the Trotter error.
This confirms that the circuit operates well outside the regime
where first-order Trotter bounds are informative, consistent with
the adiabatic analysis above.
We include this estimate for completeness; a meaningful Trotter error
analysis would require higher-order bounds beyond the scope of this
work.

After transpilation to the native gate set of the VTT Q50 processor, the circuit comprises $48$ two-qubit $CZ$ gates at depth $69$.
No published characterisation of the per-gate fidelity of the
VTT Q50 is available to the authors at the time of writing.
As a rough indication, IQM reports typical median two-qubit $CZ$
gate fidelities in the range $99.0\%$--$99.5\%$ for their
superconducting processors~\cite{IQMRadiance}; using $f = 0.993$
as a representative value gives a cumulative gate error of
approximately $1 - 0.993^{48} \approx 28\%$ over $48$ gates,
meaning that a significant fraction of shots may be corrupted
by gate errors alone before readout errors are accounted for.
We emphasise that this is an order-of-magnitude estimate only.

\begin{table}[b!]
\centering
\caption{Comparison of simulated and experimental quantum optimization results and the Euclidean distance to the continuous optimum with $T=2$, $p=4$ and $\varepsilon=0.05$.}
\label{tab:sim_vs_hw}
\begin{tabular}{lcccc}
\hline
 & \multicolumn{2}{c}{Simulation}
 & \multicolumn{2}{c}{VTT~Q50 quantum processor} \\
\cline{2-3}\cline{4-5}
Metric & $(x_1,x_2)$ & Distance & $(x_1,x_2)$ & Distance \\
\hline

Cont. optimum
& $(0.9697,-0.8970)$ & 0
& $(0.9697,-0.8970)$ & 0 \\

Optimal solution
& $(0.9762,-0.8968)$ & 0.0065
& $(0.9762,-0.8968)$ & 0.0065 \\

MAP estimate
& $(1.0873,-0.9127)$ & 0.1187
& $(1.0238,-0.8651)$ & 0.0628 \\

Sample mean
& $(0.9983,-0.9766)$ & 0.0845
& $(0.9494,-1.0030)$ & 0.1080 \\

\hline
\end{tabular}
\end{table}

\ref{tab:sim_vs_hw} summarises the key numerical estimates for both
the noiseless simulation and the VTT~Q50 run.
In both cases, the best-by-cost solution achieves
$\hat{f}=(0.9762,-0.8968)$, deviating from $f^*$ by $\|\hat{f}-f^*\|
\approx 0.007$, which is well within one grid spacing $\Delta\approx
0.016$.

\begin{figure*}[t]
\centering
\graphicspath{{Figures/}}

\setlength{\tabcolsep}{2pt}
\renewcommand{\arraystretch}{0}

\newcommand{\reconw}{0.24\textwidth}

\begin{tabular}{cc}
Simulation& 
VTT~Q50 quantum processor  \\
\includegraphics[width=0.24\linewidth]{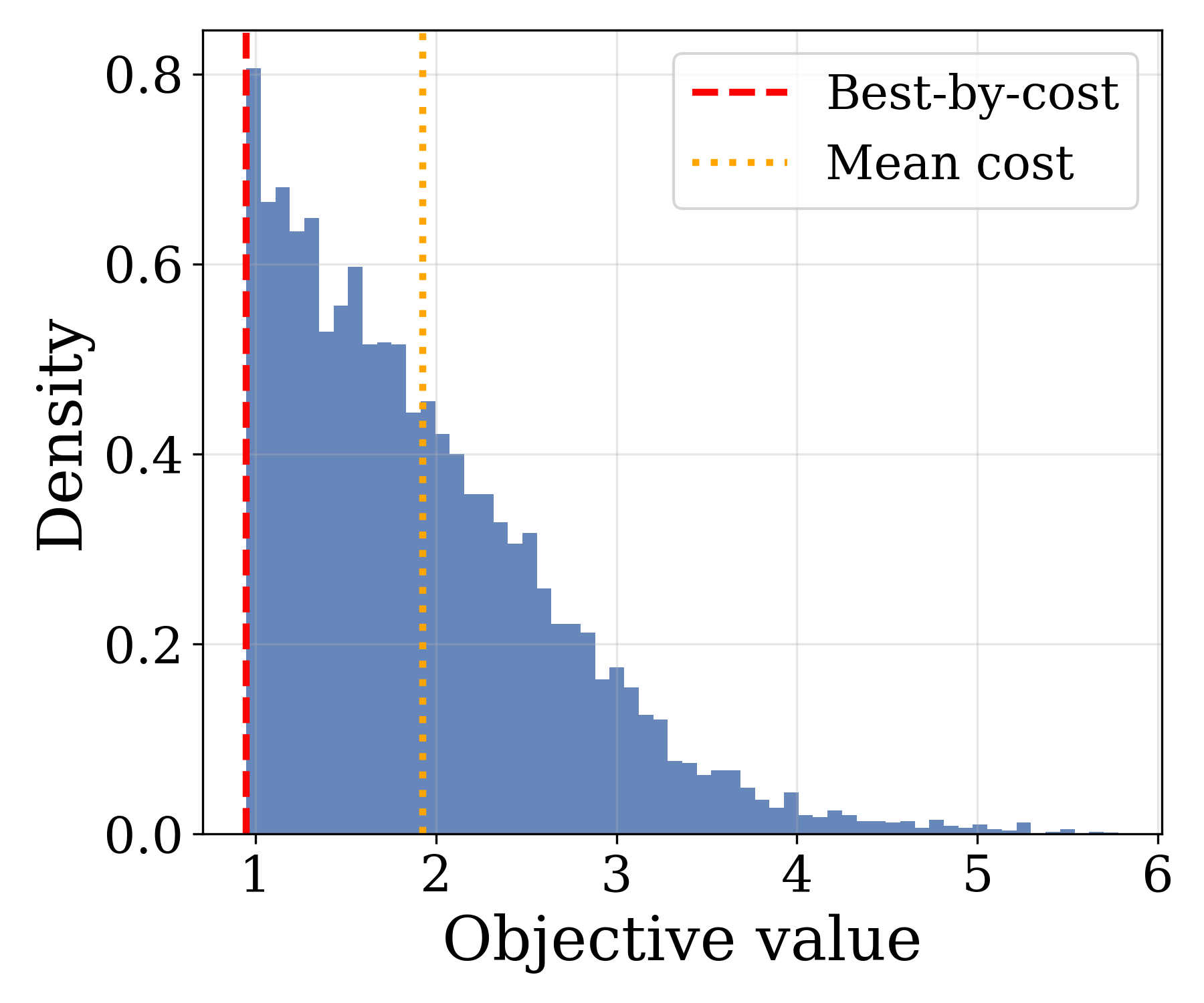} &
\includegraphics[width=0.24\linewidth]{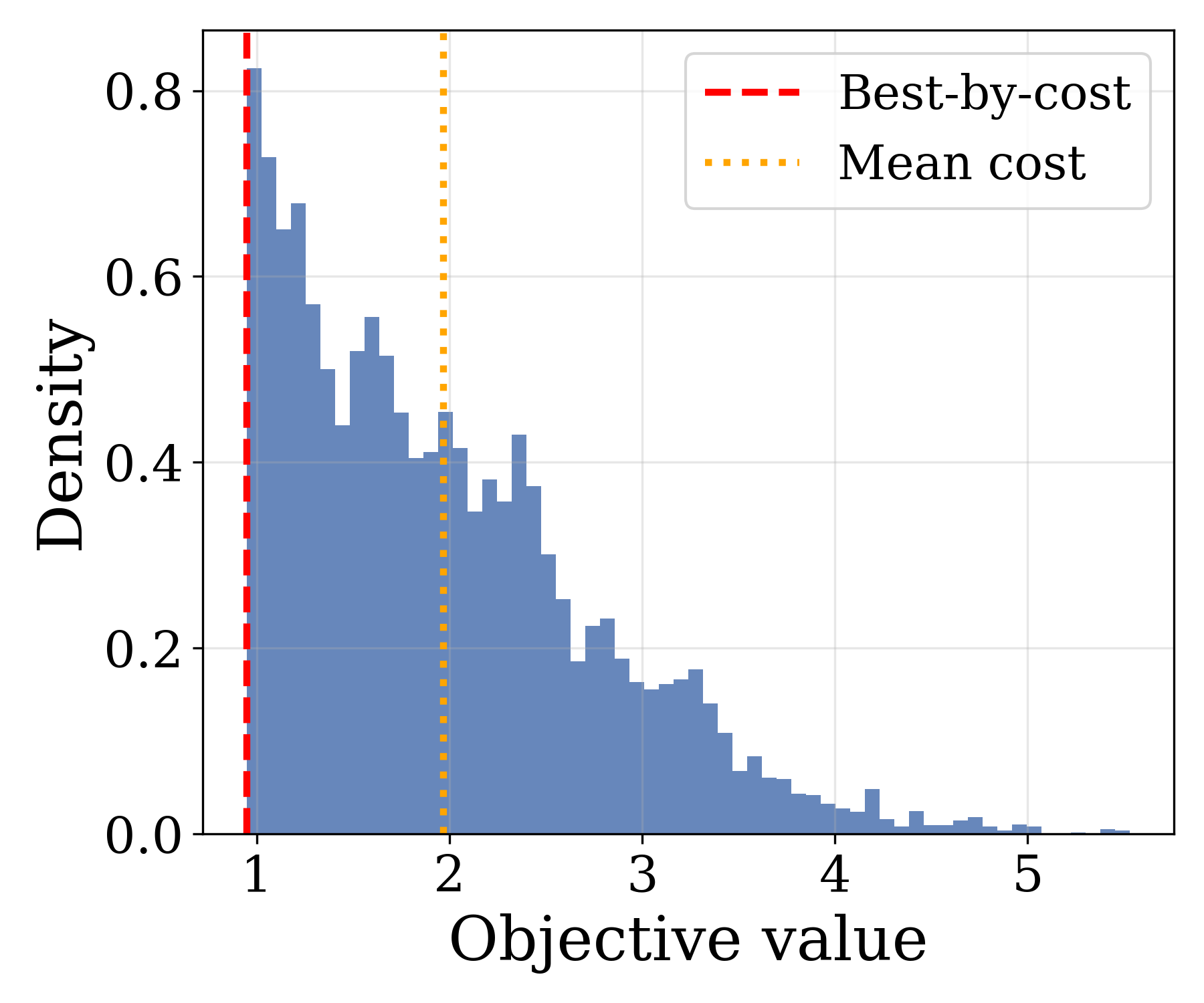}
\end{tabular}
\caption{Cost distributions for simulated and experimental results for the inverse problem in \eqref{eq:scalar3x2} with $T=2$, $p=4$ and $\varepsilon=0.05$. }
\label{fig:cost3x2}
\end{figure*}

\begin{figure*}[t]
\centering
\graphicspath{{Figures/}}

\setlength{\tabcolsep}{2pt}
\renewcommand{\arraystretch}{0}

\newcommand{\reconw}{0.24\textwidth}

\begin{tabular}{cc}
Simulation& 
VTT~Q50 quantum processor  \\[6pt]
\subcaptionbox{$x_1$ values\label{fig:boundary_input_smooth_b}}{%
\includegraphics[width=0.24\linewidth]{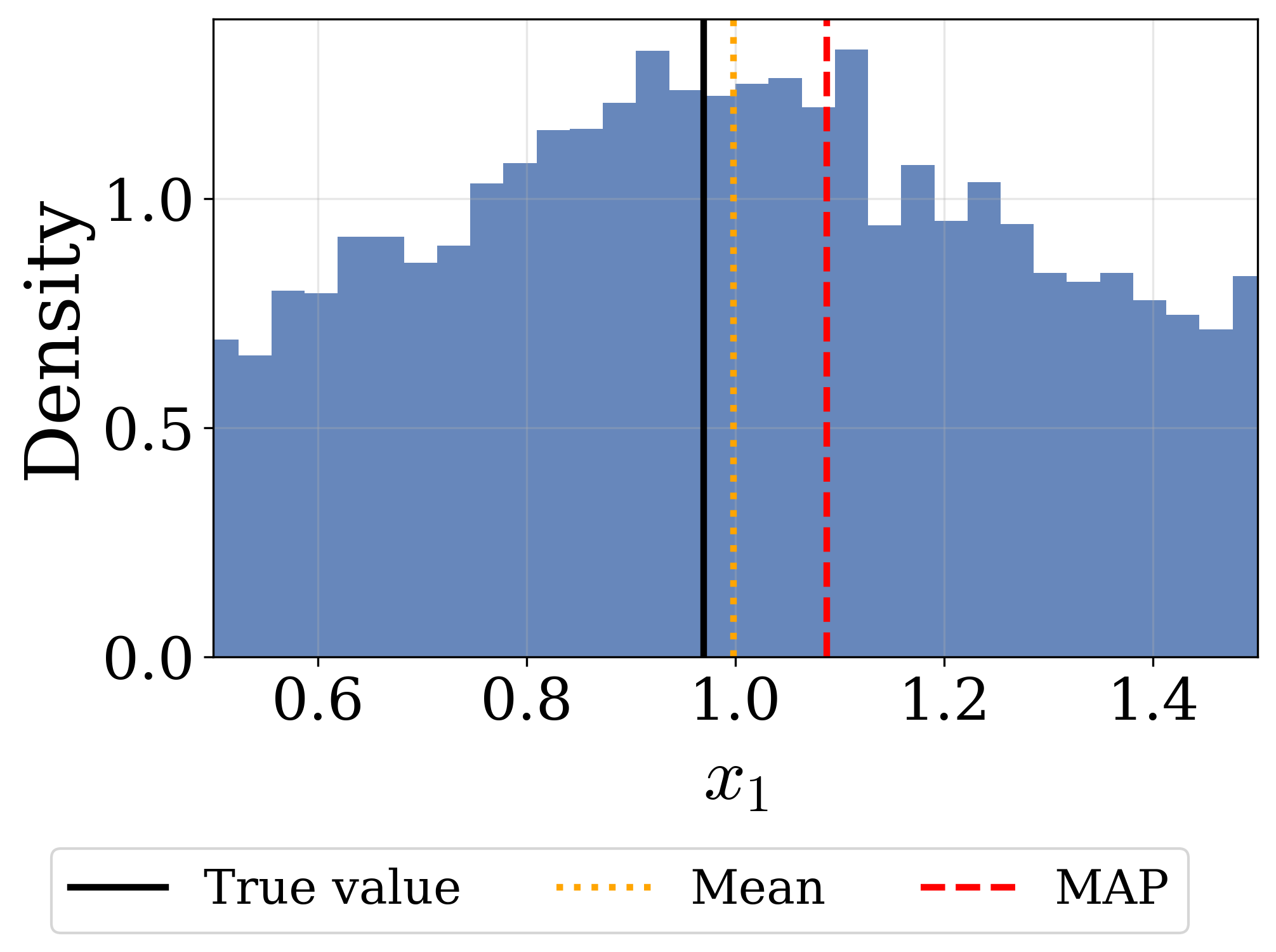}}

\subcaptionbox{$x_2$ values\label{fig:boundary_input_smooth_c}}{%
\includegraphics[width=0.24\linewidth]{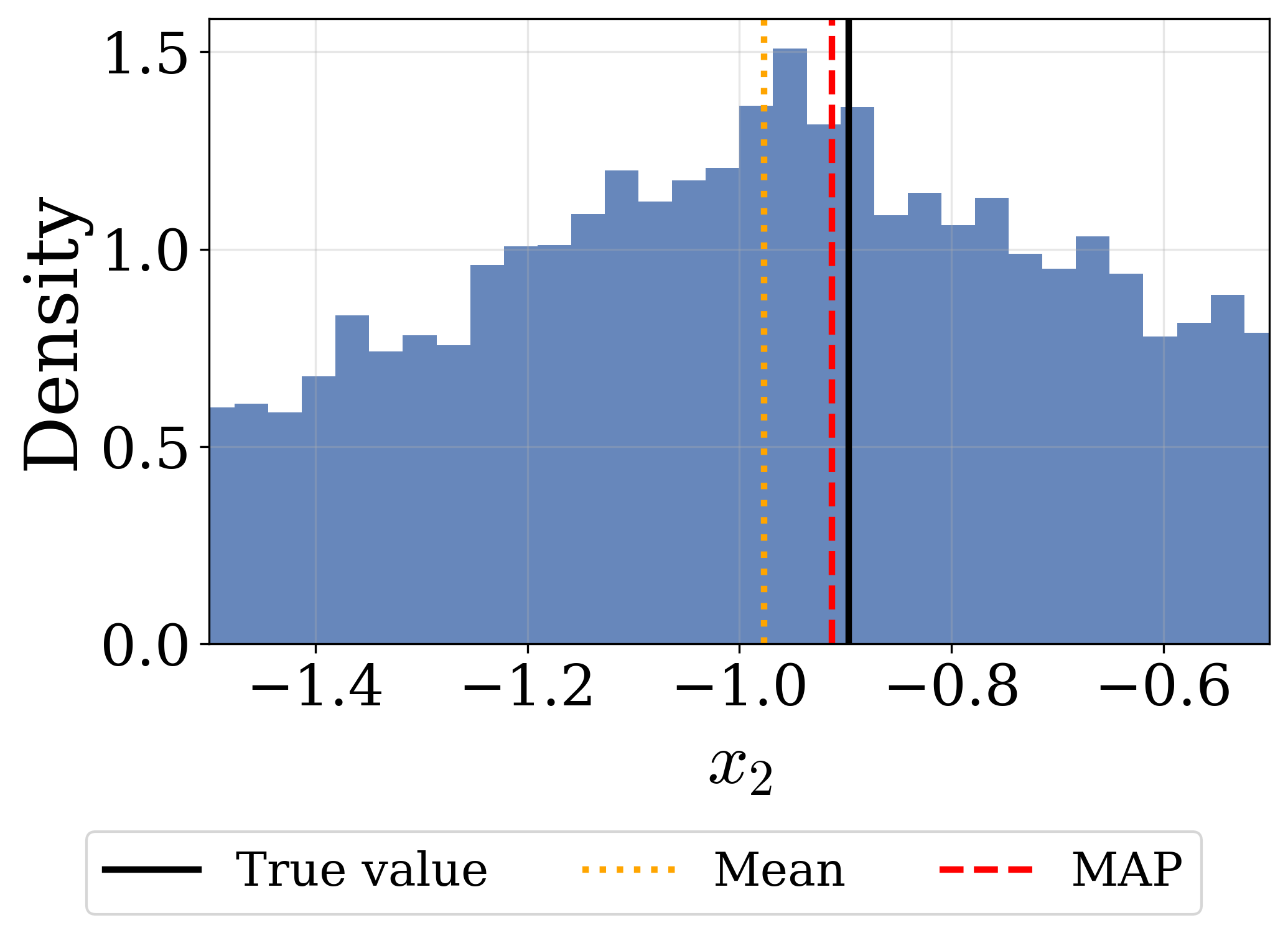}}& 

\subcaptionbox{$x_1$ values\label{fig:obj3x2_x1}}{%
\includegraphics[width=0.24\linewidth]{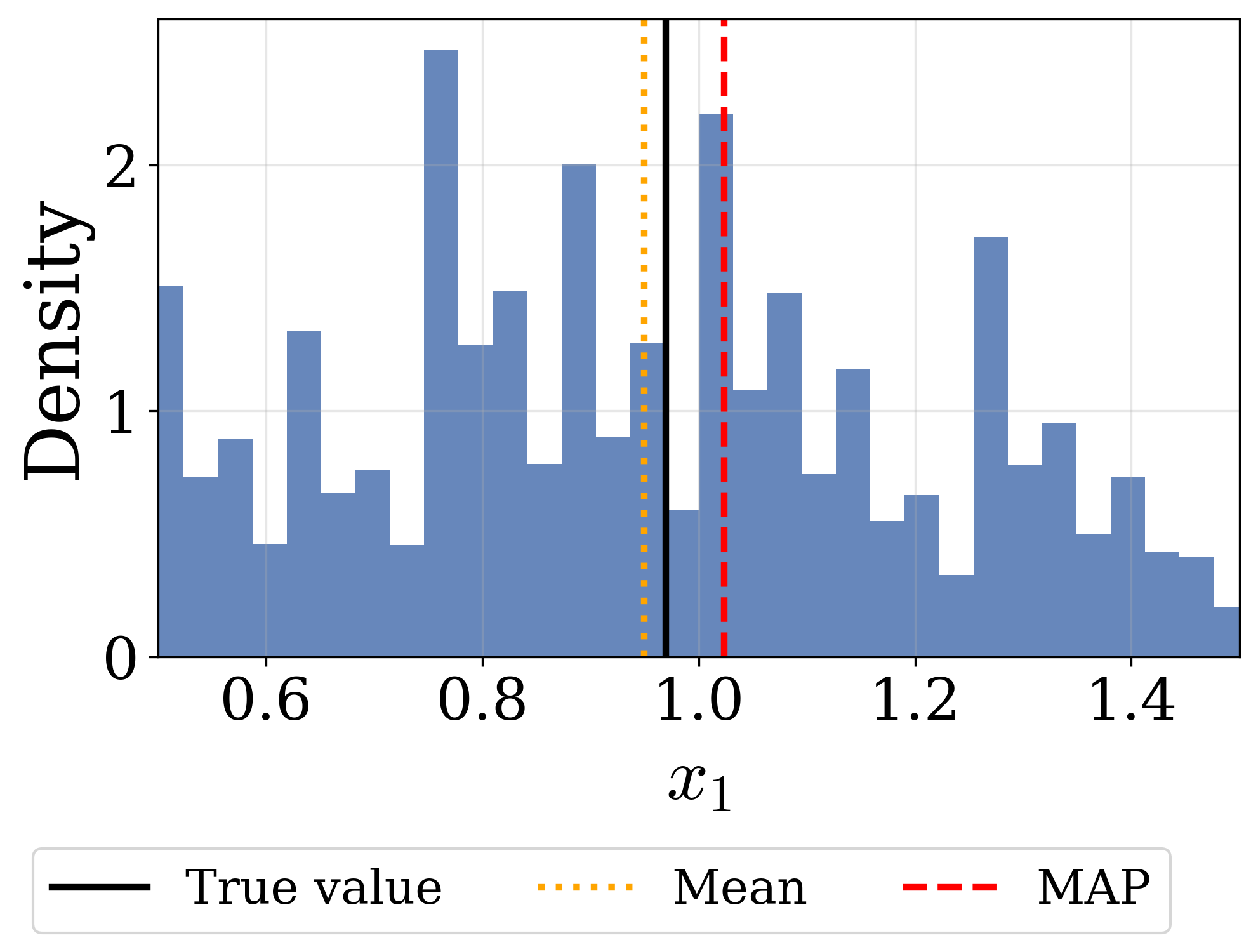}}

\subcaptionbox{$x_2$ values\label{fig:obj3x2_x2}}{%
\includegraphics[width=0.24\linewidth]{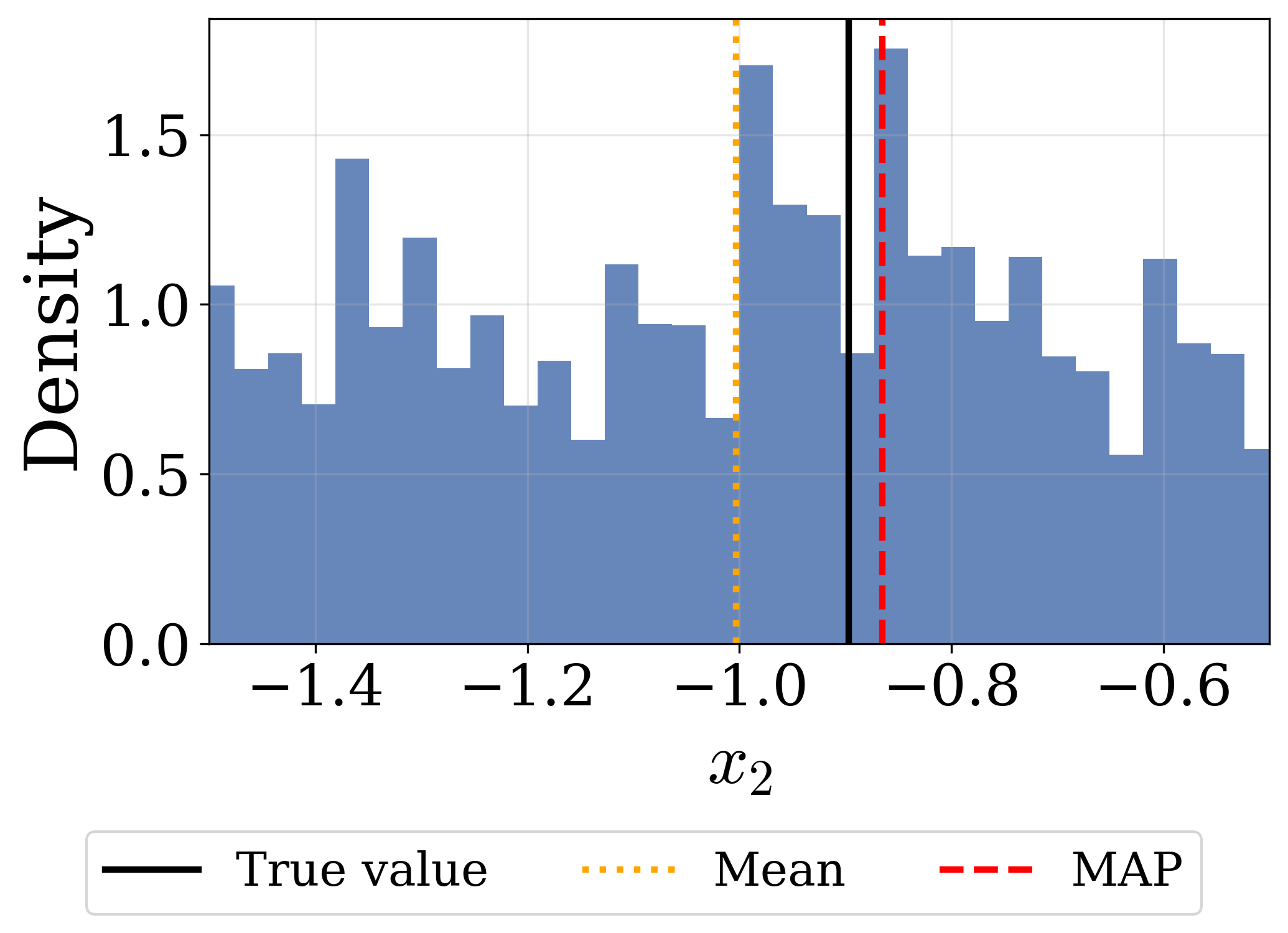}}

\end{tabular}

\caption{Histogram over samples $x_1$ and $x_2$ values for simulated and experimental results for the inverse problem in \eqref{eq:scalar3x2} with $T=2$, $p=4$ and $\varepsilon=0.05$. }
\label{fig:x1x2Histograms}
\end{figure*}

\ref{fig:cost3x2} shows the distribution of objective values over all measurement shots. In both simulation and hardware runs, the distributions are concentrated near the minimum objective value and exhibit a right tail toward higher-cost configurations, indicating that the quantum evolution preferentially samples low-energy states despite the strong violation of the adiabatic condition. The simulated distribution is more sharply concentrated, whereas the hardware distribution exhibits a heavier tail, consistent with the approximately $28\%$ cumulative gate error estimated above.


\ref{fig:x1x2Histograms} shows the marginal histograms of $f_1$ and
$f_2$.
In the noiseless simulation, both marginals display a clear unimodal
peak centred near $f^*$, reflecting genuine probability concentration.
On the hardware, the distributions are broader and exhibit an irregular, spiky structure, suggesting noise-induced spreading in which isolated bitstrings are 
amplified by hardware error patterns rather than by the cost landscape.
Despite this, both marginals remain centered in the correct region, indicating that the underlying quantum optimization dynamics remain visible despite noise.

\ref{fig:theoretical3x2} shows the eigendirections and the $20\%$, $40\%$, and $60\%$ credible regions of the continuous posterior distribution $\mathcal N(f^\star,Q^{-1})$, where $Q=A^\top A+\alpha I$. As these ellipses are derived from the unconstrained continuous problem, they extend partially outside the discrete domain $[0.5,1.5]\times[-1.5,-0.5]$ and should therefore be viewed as a reference geometry rather than as credible regions for the discretized problem.

\ref{fig:empirical3x2sim} and \ref{fig:empirical3x2qpu} show the corresponding empirical sample distributions obtained from simulation and the VTT~Q50 processor, together with the empirical eigendirections, sample mean, and the most frequently observed configurations (four in this example due to ties in the histogram counts). In both cases, the sample distributions exhibit a dominant mode near $f^\star$, and the sample mean provides a stable estimate with distances to $f^\star$ of $0.085$ (simulation) and $0.108$ (hardware), respectively (see \ref{tab:sim_vs_hw}). However, the empirical eigendirections, mean, and MAP samples do not fully align with the theoretical posterior geometry. This discrepancy is more pronounced in the hardware results, where the distribution is broader and more irregular, consistent with the noise-induced spreading discussed above and the approximately $28\%$ cumulative gate error estimated previously.

To quantify the agreement between the empirical and theoretical principal directions, we compute the angular discrepancy
$\Delta\theta=\cos^{-1}(|\boldsymbol{v}_{\rm emp}^{\top}\boldsymbol{v}_{\rm th}|)$,
where $\boldsymbol{v}_{\rm emp}$ and $\boldsymbol{v}_{\rm th}$ denote the dominant empirical and theoretical eigenvectors, respectively. We obtain $\Delta\theta=36.7^\circ$ for the simulated samples and $\Delta\theta=38.0^\circ$ for the hardware samples. While these values indicate that the empirical distributions do not reproduce the posterior geometry quantitatively, they nevertheless exhibit a consistent directional structure, suggesting that the sampled distributions retain some information about the underlying inverse problem.

Taken together, the results show that the correct minimizer can be recovered through post-selection in both simulation and hardware experiments. In contrast, the empirical mean, MAP estimates, and sample geometry are more strongly affected by finite-time evolution, Trotterization, and hardware noise, leading to noticeable deviations from the theoretical posterior structure.











\begin{figure*}[t]
\centering
\graphicspath{{Figures/}{images/2d-new/}}

\setlength{\tabcolsep}{4pt}
\renewcommand{\arraystretch}{0}

\begin{tabular}{ccc}
Theoretical posterior &
Simulation &
VTT~Q50 quantum processor \\[6pt]

\subcaptionbox{\label{fig:theoretical3x2}}{%
\includegraphics[width=0.26\linewidth]{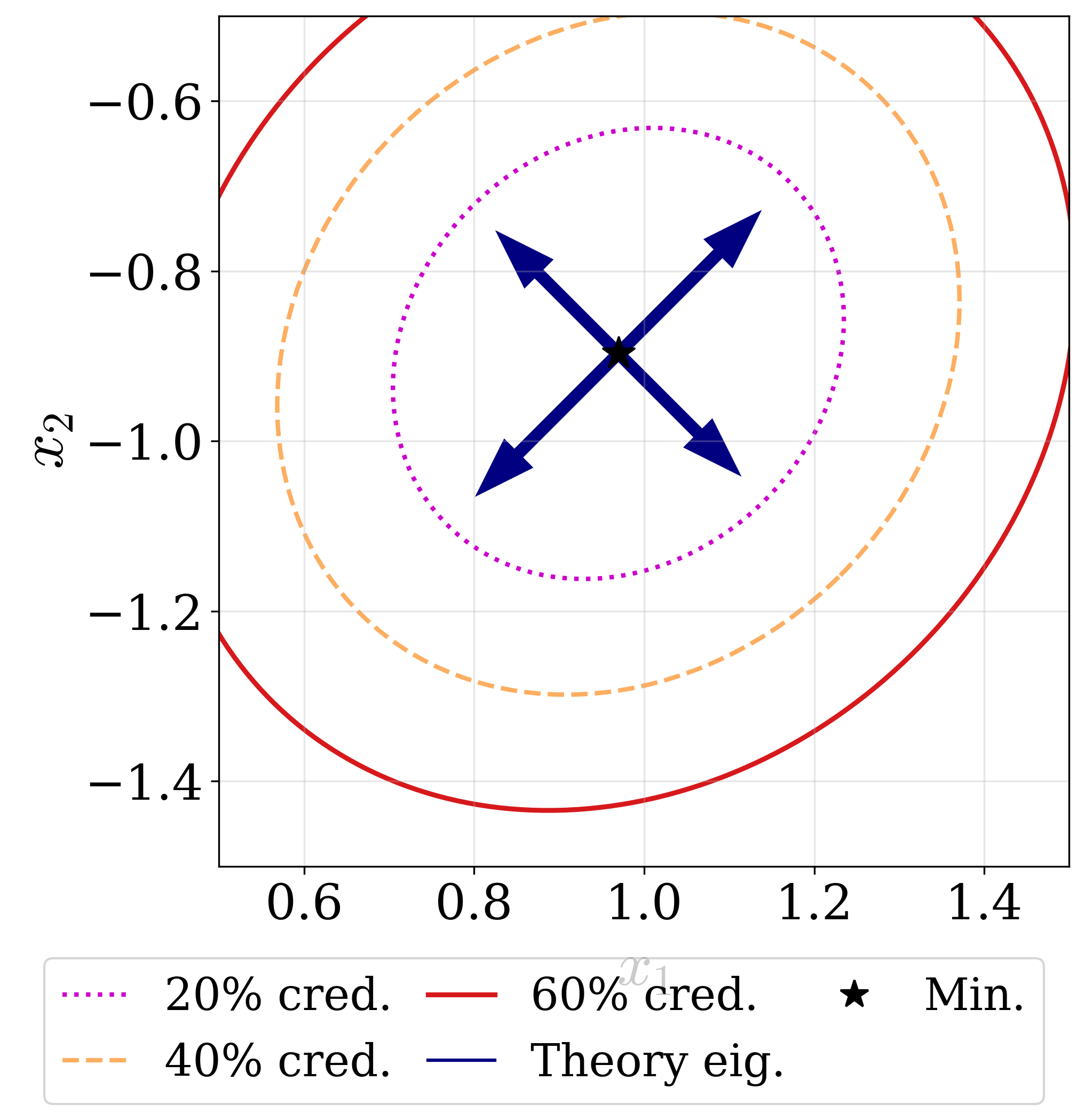}}
&
\subcaptionbox{\label{fig:empirical3x2sim}}{%
\includegraphics[width=0.29\linewidth]{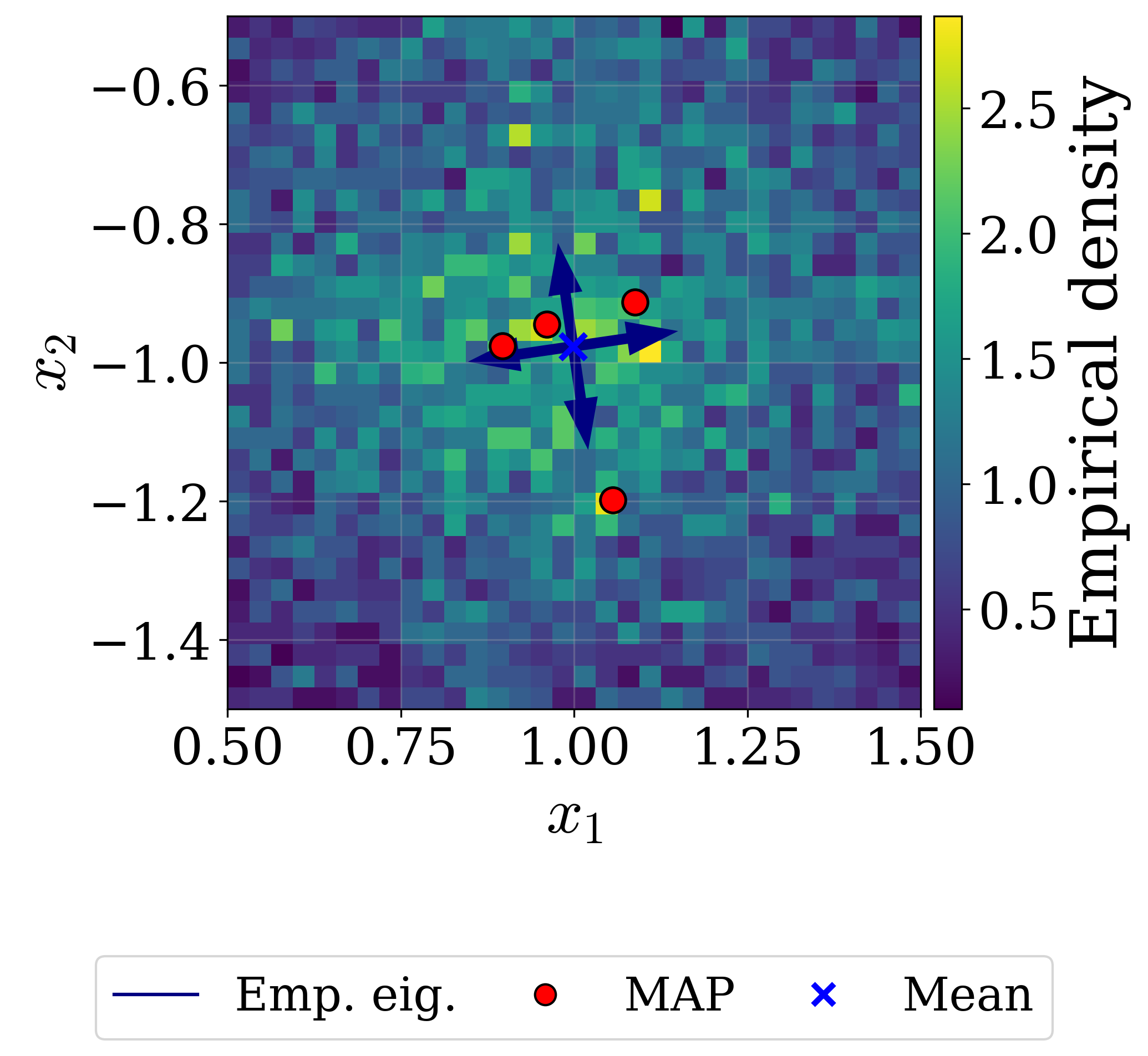}}
&
\subcaptionbox{\label{fig:empirical3x2qpu}}{%
\includegraphics[width=0.29\linewidth]{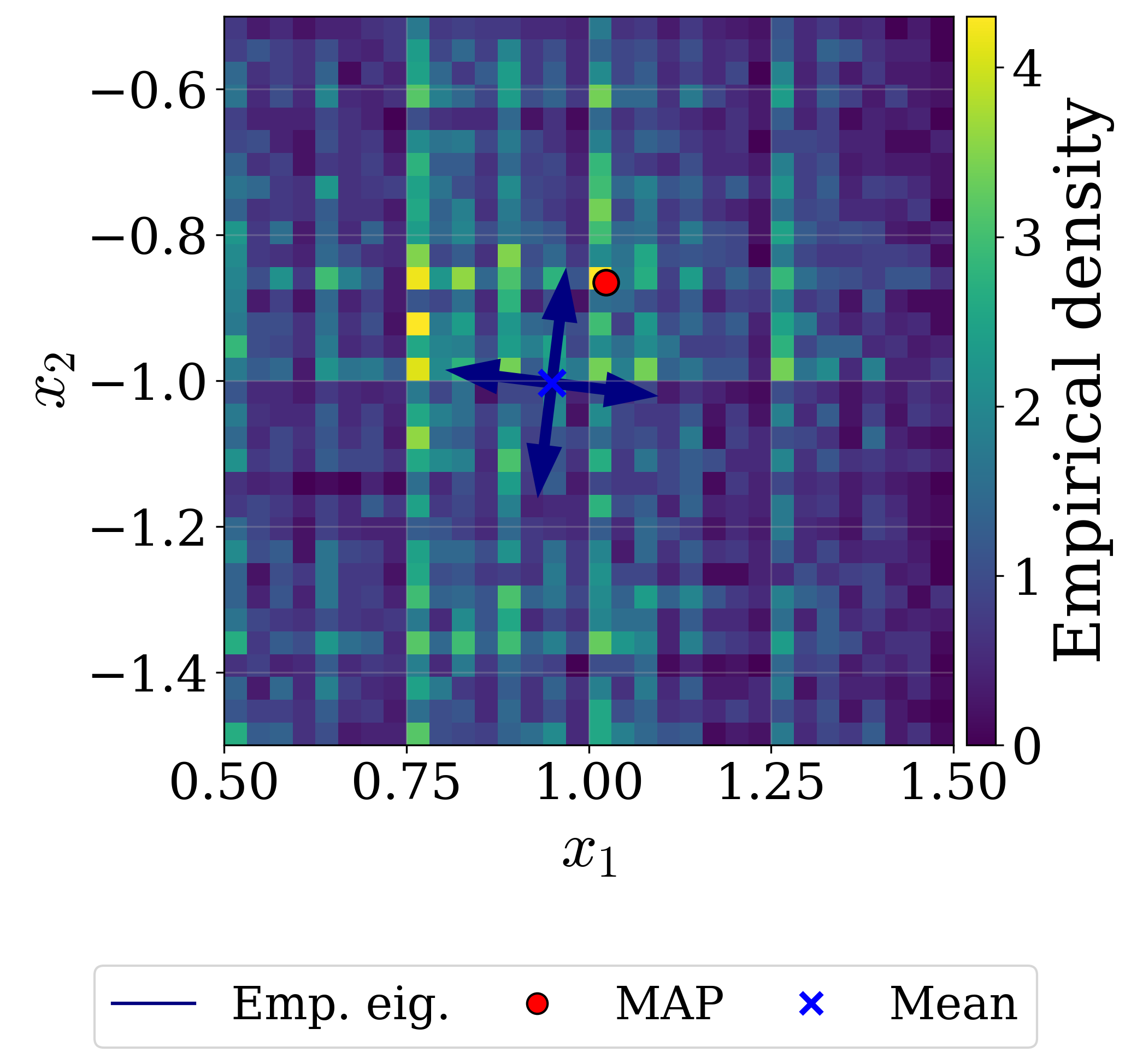}}

\end{tabular}

\caption{Theoretical posterior geometry and empirical sample distributions for the inverse problem \eqref{eq:scalar3x2} with $T=2$, $p=4$, and $\varepsilon=0.05$. The left panel shows the eigendirections and credible regions of the continuous posterior, while the center and right panels show the corresponding sample distributions obtained from simulation and the VTT~Q50 quantum processor, respectively.}
\label{fig:objective3x2}
\end{figure*}







\begin{figure}[t]
  \centering
  \begin{tabular}{cc}
\includegraphics[width=0.40\linewidth]{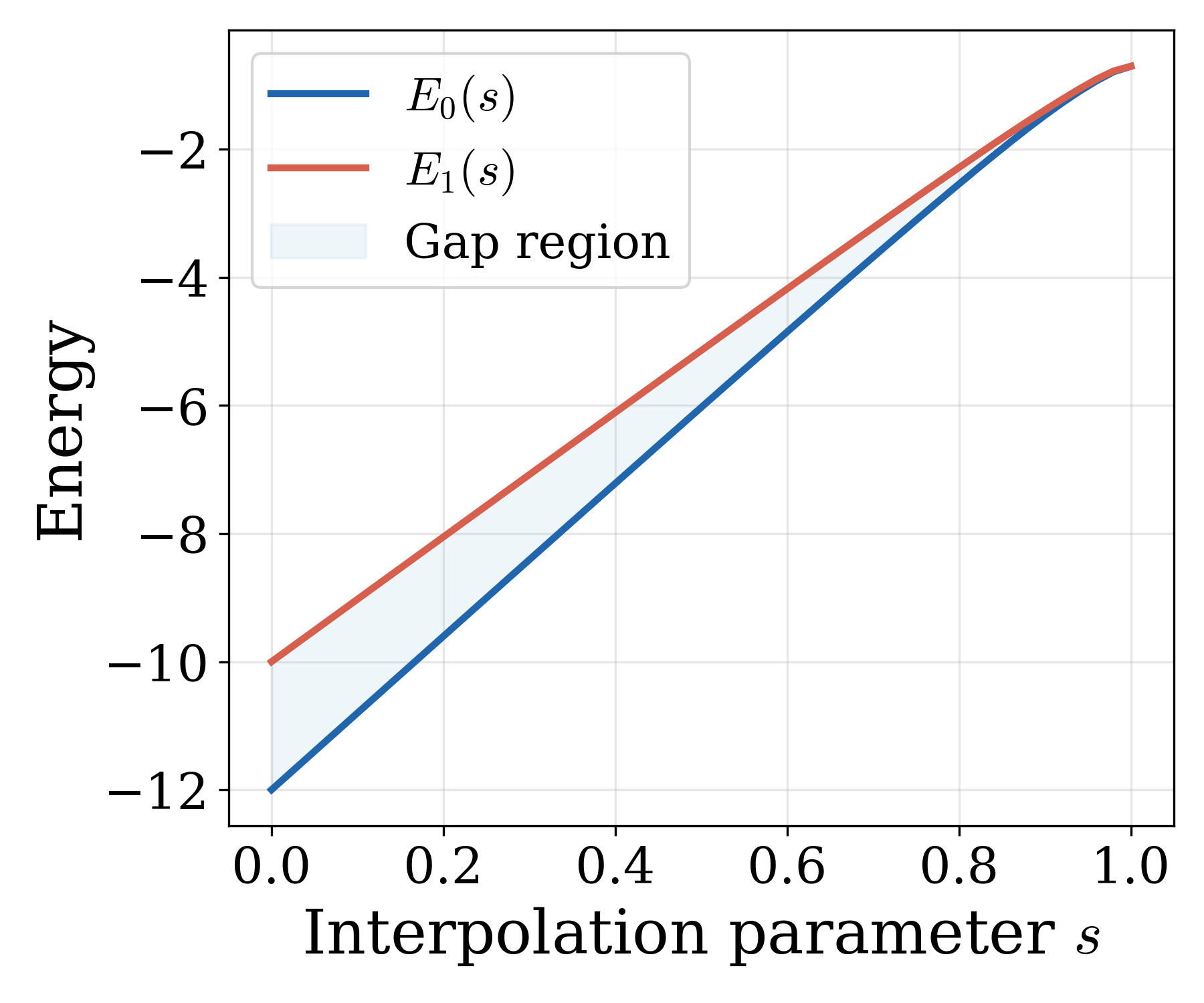} &
\includegraphics[width=0.40\linewidth]{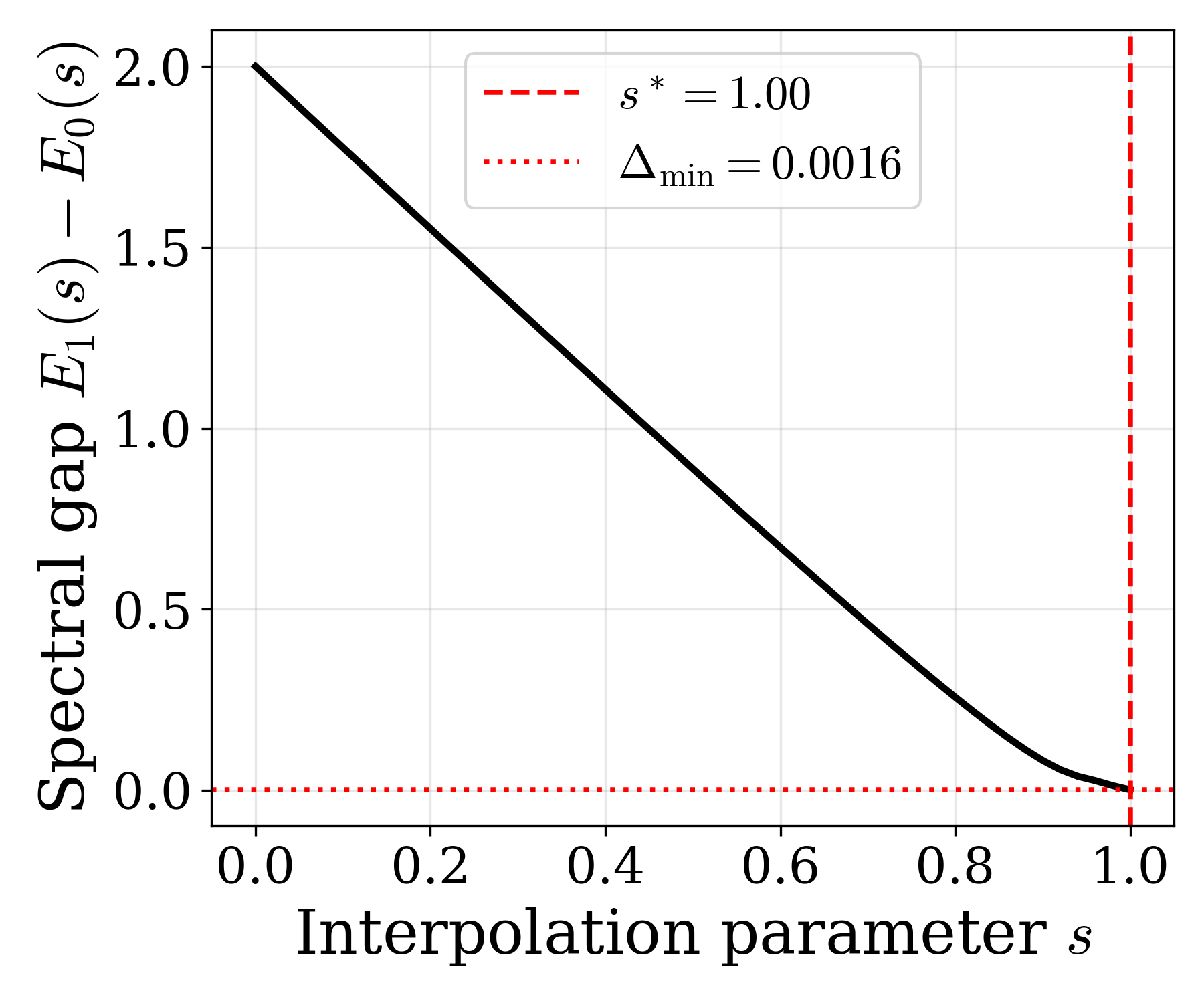}
\end{tabular}
  \caption{Spectral gap $\gamma(s)=E_1(s)-E_0(s)$ of the interpolating
           Hamiltonian $H(s)=(1-s)H_B+sH_C$ for
           problem~\eqref{eq:scalar3x2} ($n=12$ qubits,
           $\varepsilon=0.05$).
           Left: ground and first excited state energies $E_0(s)$ and
           $E_1(s)$.
           Right: spectral gap, which decreases monotonically from
           $\gamma(0)=2$ to $\Delta_{\min}=0.001020$ at $s=1$,
           implying an adiabatic time requirement of order $10^7$,
           far exceeding the chosen $T=2$.}
  \label{fig:spectral_gap}
\end{figure}

\subsection{Optimization with $\ell^1$-regularization in Haar-wavelet basis} \label{sec:results3}

We next perform experiments on the VTT~Q50 quantum computer to
recover five Haar wavelet coefficients of a blurred one-dimensional
signal, using $k \in \{2, 3, 4\}$ bits per variable.
The circuit parameters $T=4$, $p=3$, and $\varepsilon=0.03$ were
selected by running a small number of QPU jobs with $k=4$
setting and choosing the combination that gave the
best-by-cost reconstruction.
The same parameters are then used unchanged across all three
bit settings to enable a direct comparison. The forward operator $A\in\mathbb{R}^{64\times 64}$ is a
convolution with a Gaussian kernel with standard deviation $\sigma=2.0$ \cite{kaipio2005statistical}, and the regularization parameter is $\alpha=0.05$. We aim to solve the inverse problem
\begin{equation}\label{eq:haar}
\min_{\boldsymbol{x}\in\mathbb{R}^5}\; \|A\Phi\boldsymbol{x} - \boldsymbol{g}\|_2^2
    + \alpha\,\|\boldsymbol{x}\|_1,
\end{equation}
where $\Phi\in\mathbb{R}^{64\times 5}$ collects the five Haar basis
vectors as columns, and
$\boldsymbol{g}=A\Phi\boldsymbol{x}_{\rm true}+\boldsymbol{\delta}$
is the observed blurry and noisy signal with noise level
$\|\boldsymbol{g}_{\rm true}-\boldsymbol{g}\|_2/
\|\boldsymbol{g}_{\rm true}\|_2=0.02$, where
$\boldsymbol{g}_{\rm true}$ denotes the noise-free signal.

The reconstruction is represented in the Haar basis using the five
wavelets where, for $N=64$,
\[
\begin{aligned}
\psi_{j,k}(i)
&=
\begin{cases}
\phantom{-}2^{j/2}N^{-1/2},
& k\frac{N}{2^j}\le i <
\left(k+\frac12\right)\frac{N}{2^j},\\[1ex]
-2^{j/2}N^{-1/2},
&
\left(k+\frac12\right)\frac{N}{2^j}
\le i < (k+1)\frac{N}{2^j},\\[1ex]
0,
& \text{otherwise},
\end{cases}\\ \Phi&:=\{\psi_{0,0},
\psi_{1,0},
\psi_{1,1},
\psi_{2,1},
\psi_{2,2}\},
\end{aligned}
\]
for $i=0,\dots,N-1$.
The observed data
$\boldsymbol{g}=A\Phi\boldsymbol{x}_{\rm true}+\boldsymbol{\delta}$
is a blurred and noisy version of the true signal
$\Phi\boldsymbol{x}_{\rm true}$.

The $\ell^1$ regularization promotes sparsity in the coefficient vector
$\boldsymbol{x}$ and is implemented using the split-variable
construction of \ref{sec:qvar}, namely
$\boldsymbol{x}=\boldsymbol{x}^+-\boldsymbol{x}^-$ with
$\boldsymbol{x}^+\ge 0$ and $\boldsymbol{x}^-\ge 0$.
The penalty
$\alpha\,\mathbf 1^\top(\boldsymbol{x}^++\boldsymbol{x}^-)$
coincides with $\alpha\|\boldsymbol{x}\|_1$ at the optimum, where for
each component at most one of $x_\ell^+$ and $x_\ell^-$ is nonzero.

This yields $d_{\rm split}=2\cdot 5=10$ nonnegative variables. Using
$k\in\{2,3,4\}$ bits per variable (cf. \ref{sec:qvar}) results in a total qubit count of
$2kd = 2k\cdot 5 \in \{20,30,40\}$.

We note that the three experiments solve slightly different discretized
problems. For each value of $k$, the true coefficient vector
$\boldsymbol{x}_{\rm true}$ is chosen to lie exactly on the
corresponding grid, whose spacing is $\delta = 1/(2^k-1)$.
Consequently,
$\boldsymbol{x}_{\rm true}$ takes the values
\[
\left(\frac{1}{3},-\frac{2}{3},\frac{1}{3},\frac{2}{3},\frac{2}{3}\right)^\top,
\quad
\left(\frac{1}{3},-\frac{2}{3},\frac{1}{3},\frac{2}{3},\frac{2}{3}\right)^\top,
\quad
\left(\frac{2}{15},-\frac{3}{15},\frac{3}{15},\frac{13}{15},\frac{13}{15}\right)^\top,
\]
for $k=2,3,4$, respectively.
This ensures that the true solution is exactly representable on each
grid, thereby eliminating discretization error as a confounding factor
in the comparison.

 
\ref{tab:haar} summarises the circuit and result statistics for
all three runs.
The number of $CZ$ gates grows substantially with qubit count:
from $1184$ ($k=2$, 20 qubits) to $1355$ ($k=3$, 30 qubits)
to $1371$ ($k=4$, 40 qubits), while the active pair count after
sparsification increases from $94$ to $105$ to $111$.
The cumulative gate error, estimated using the same independent-error
model as in \ref{sec:results2}, ranges from approximately
$1-0.993^{1184}\approx 100\%$ for all three cases. This confirms that
these circuits are deep enough that the independent-error bound
is vacuous, and the algorithm operates purely as a heuristic sampler.
 
\begin{table}[b!]
\centering
\caption{Circuit and result statistics for the Haar wavelet
         experiment ($T=4$, $p=3$, $\varepsilon=0.03$,
         $N_s=10{,}000$ shots) on the VTT~Q50.}
\label{tab:haar}
\begin{tabular}{lccc}
\toprule
 & $k=2$ & $k=3$ & $k=4$ \\
\midrule
Qubits                     & 20    & 30    & 40    \\
Active pairs               & 94    & 105   & 111   \\
$CZ$ gates (transpiled)    & 1184  & 1355  & 1371  \\
Circuit depth              & 963   & 832   & 637   \\
Grid spacing $\Delta x$    & $1/3$ & $1/7$ & $1/15$ \\
Unique bitstrings sampled  & 9851  & 10000 & 10000 \\
Best-by-cost               & 0.252 & 0.241 & 0.260 \\
Mean cost                  & 2.409 & 1.913 & 1.915 \\
\bottomrule
\end{tabular}
\end{table}

\begin{figure*}[b!]
\centering
\graphicspath{{Figures/}}

\setlength{\tabcolsep}{2pt}
\renewcommand{\arraystretch}{0}

\newcommand{\reconw}{0.24\textwidth}

\begin{tabular}{c}
\\[6pt]
\subcaptionbox{2 bits per variable\label{fig:2bits}}{%
\includegraphics[width=0.32\linewidth]{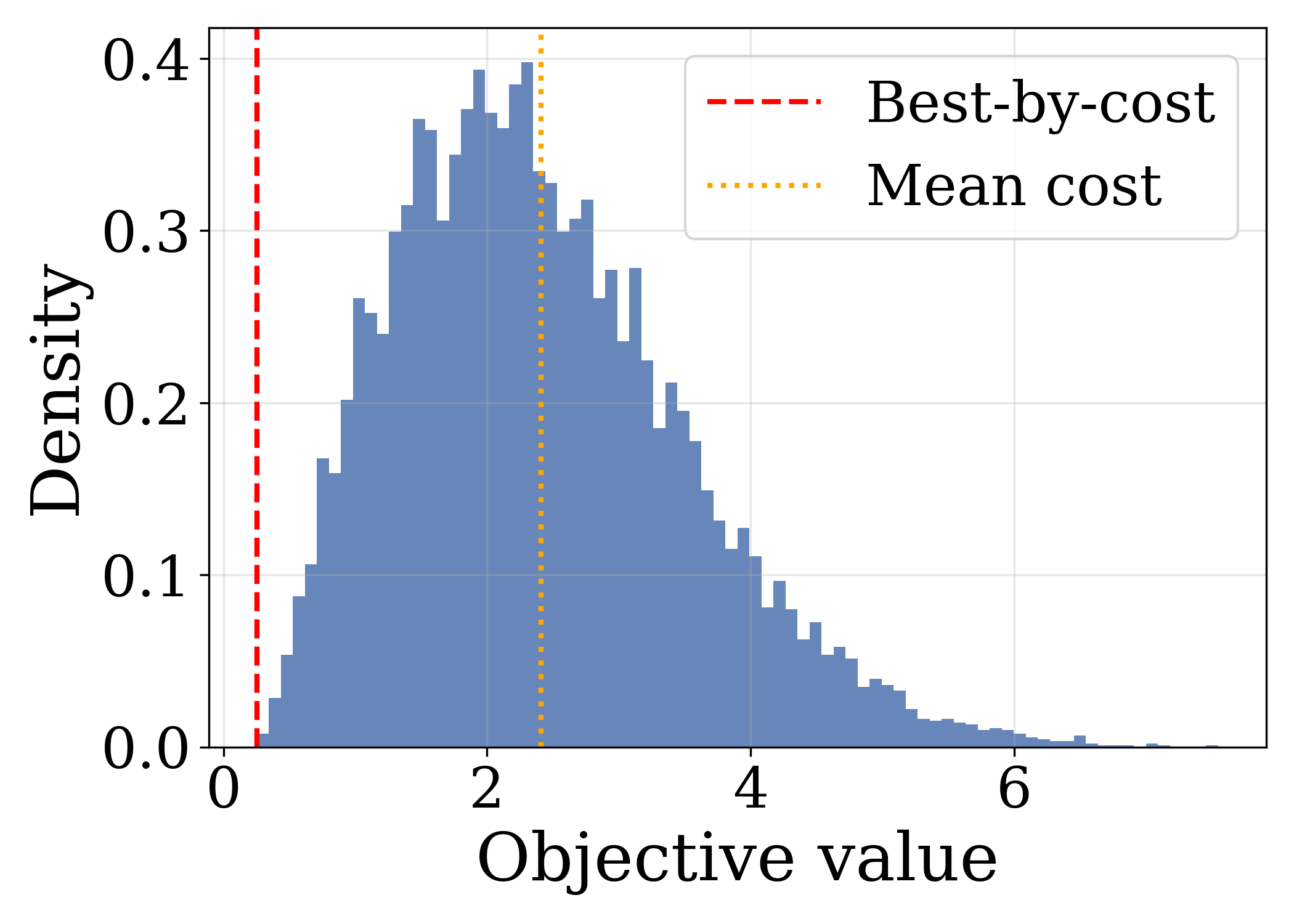}}
\subcaptionbox{3 bits per variable\label{fig:3bits}}{%
\includegraphics[width=0.32\linewidth]{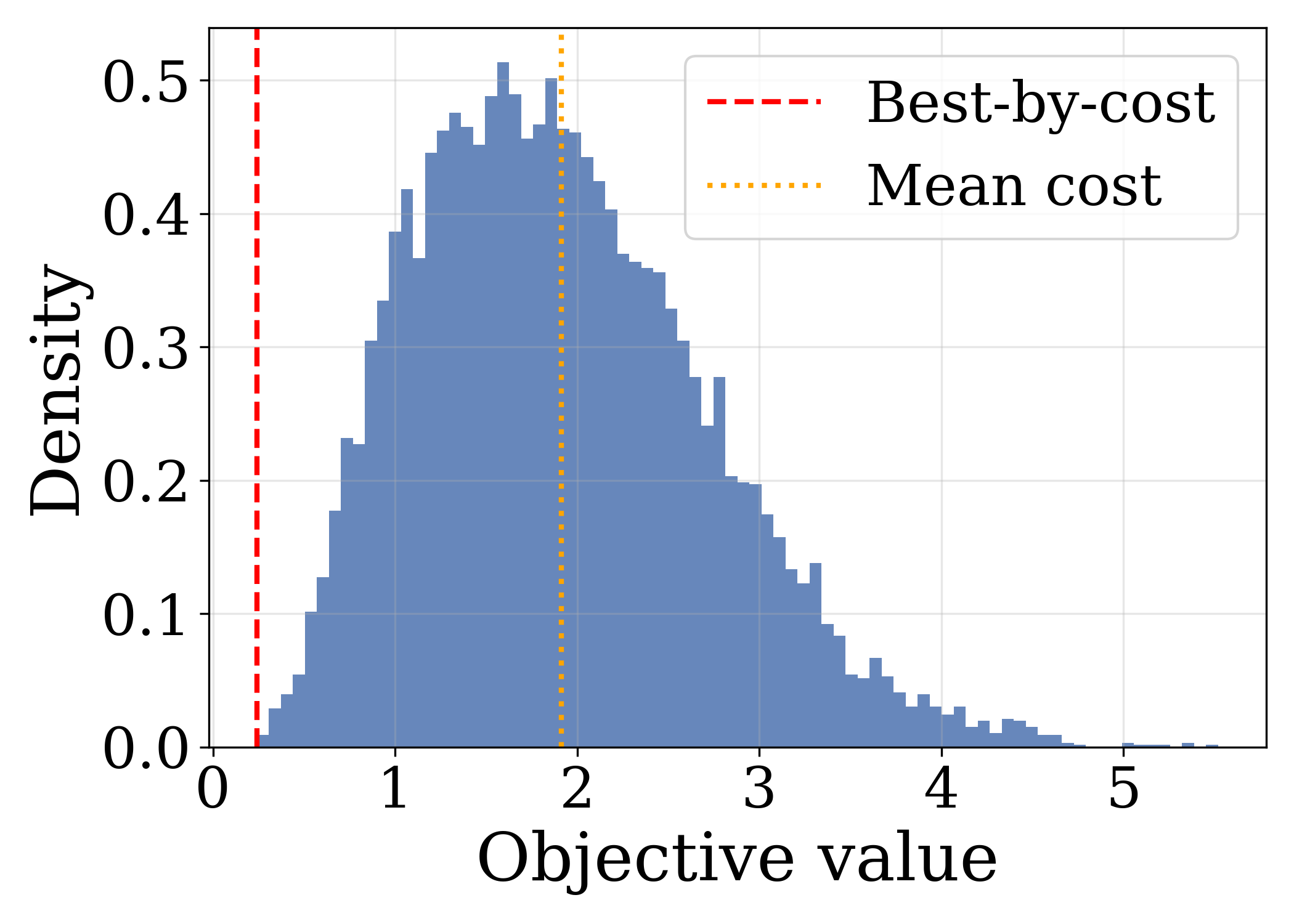}}
\subcaptionbox{4 bits per variable\label{fig:4bits}}{%
\includegraphics[width=0.32\linewidth]{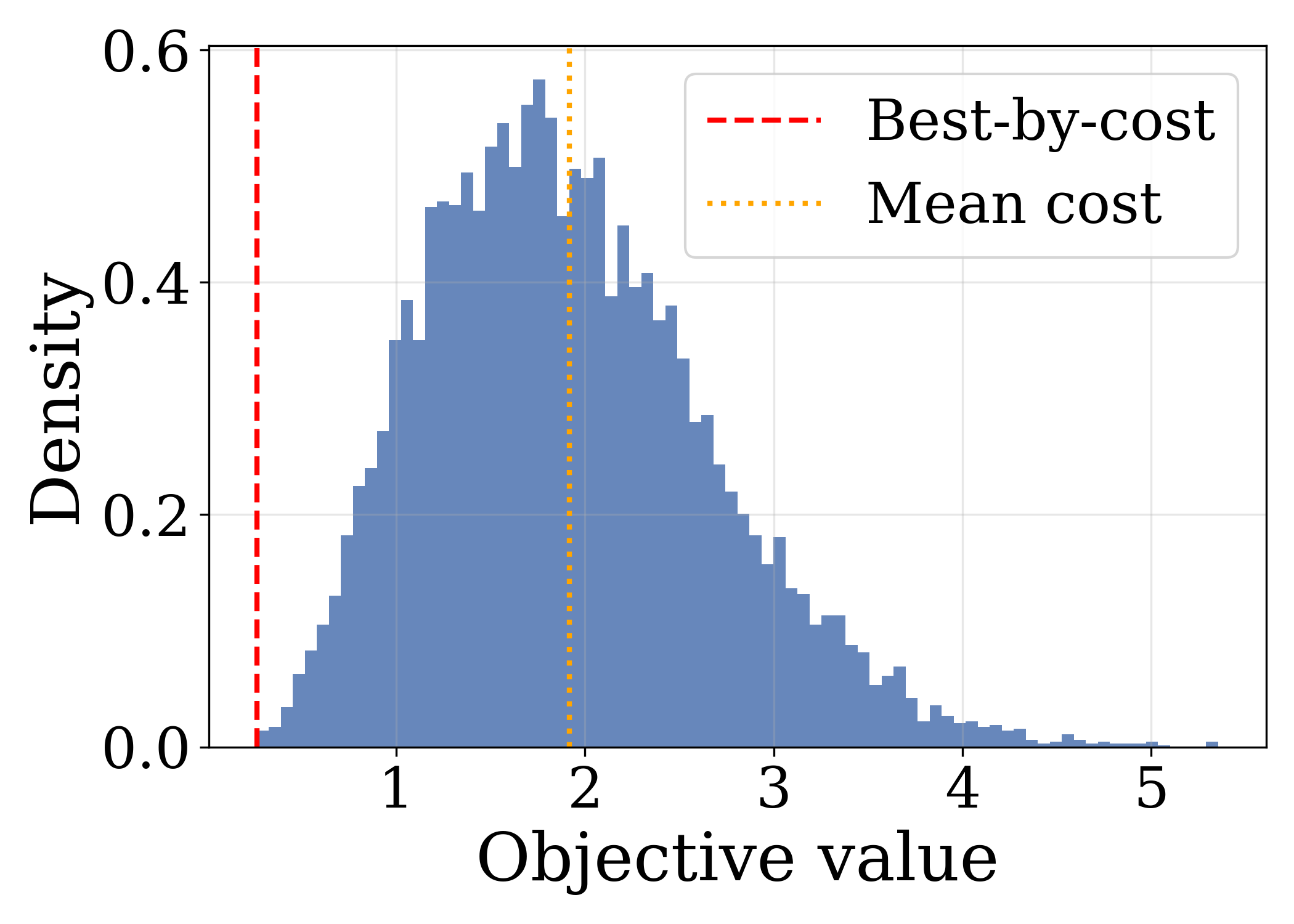}}
\end{tabular}

\caption{Cost distributions for varying bits per variable for solving the inverse problem in \eqref{eq:haar} with 5 wavelet coefficients on the VTT~Q50 Quantum computer with $T=4$, $p=3$ and $\varepsilon=0.03$. }
\label{fig:haarcost}
\end{figure*}
 
\ref{fig:haarcost} shows the distribution of sampled objective
values for each bit setting.
All three distributions are right-skewed with a sharp rise near
zero and a heavy tail, consistent with the noise-dominated regime
characterised above.
The best-by-cost values are $0.252$, $0.241$, and $0.260$ for
$k = 2, 3, 4$ respectively, showing that the three settings
perform comparably in terms of the minimum achieved cost.
The mean cost decreases from $2.4$ ($k=2$) to approximately
$1.9$ ($k=3,4$), suggesting that the 3- and 4-bit distributions
concentrate slightly more probability in the low-cost region,
likely because the narrower grid spacing reduces the energy range
of the QUBO landscape.
An important observation is that the 4-bit circuit is
\emph{shallower} (depth 637) than the 3-bit circuit (depth 832)
despite having more qubits; this is a consequence of the
$\varepsilon=0.03$ sparsification interacting differently with
the two QUBO structures after transpilation to the Q50 connectivity.


\ref{fig:haarcoeffs} shows the true, best-by-cost, and mean
wavelet coefficients for each bit setting, and \ref{tab:haarcoeffs}
summarises the Euclidean distances between the recovered and true
coefficient vectors.

Several observations can be made from \ref{fig:haarcoeffs} and
\ref{tab:haarcoeffs}.
First, the signs of all five best-by-cost coefficients agree with the
true signs in all three settings, including the negative coefficient
associated with $\psi_{1,0}$. This is nontrivial in the split-variable
formulation $x_\ell=x_\ell^+-x_\ell^-$, where negative coefficients are
represented by $x_\ell^- > x_\ell^+$.

\begin{figure*}[t!]
\centering
\graphicspath{{Figures/}}

\setlength{\tabcolsep}{2pt}
\renewcommand{\arraystretch}{0}

\newcommand{\reconw}{0.24\textwidth}

\begin{tabular}{c}
\\[6pt]
\subcaptionbox{2 bits per variable\label{fig:2bits2}}{%
\includegraphics[width=0.32\linewidth]{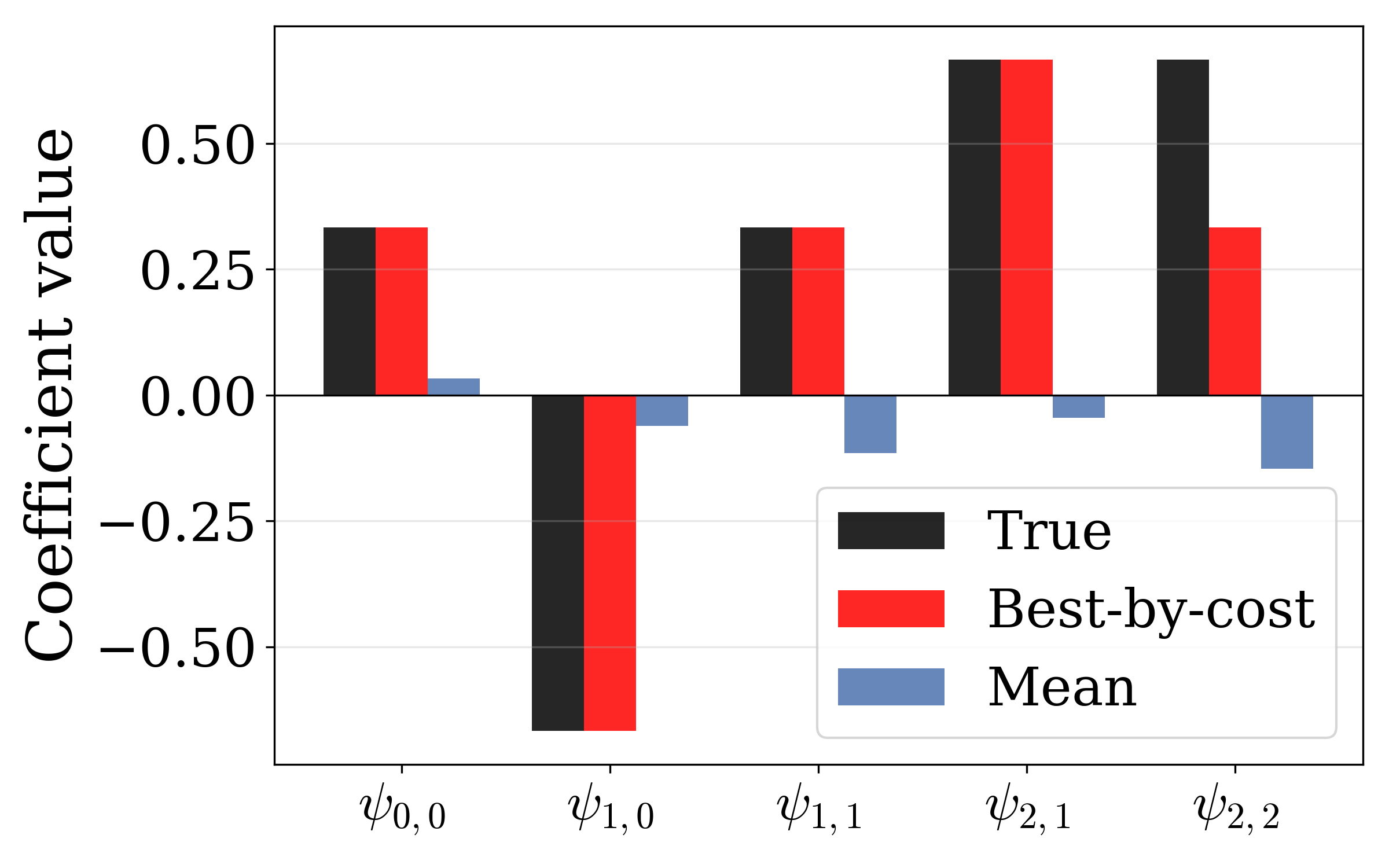}}
\subcaptionbox{3 bits per variable\label{fig:3bits2}}{%
\includegraphics[width=0.32\linewidth]{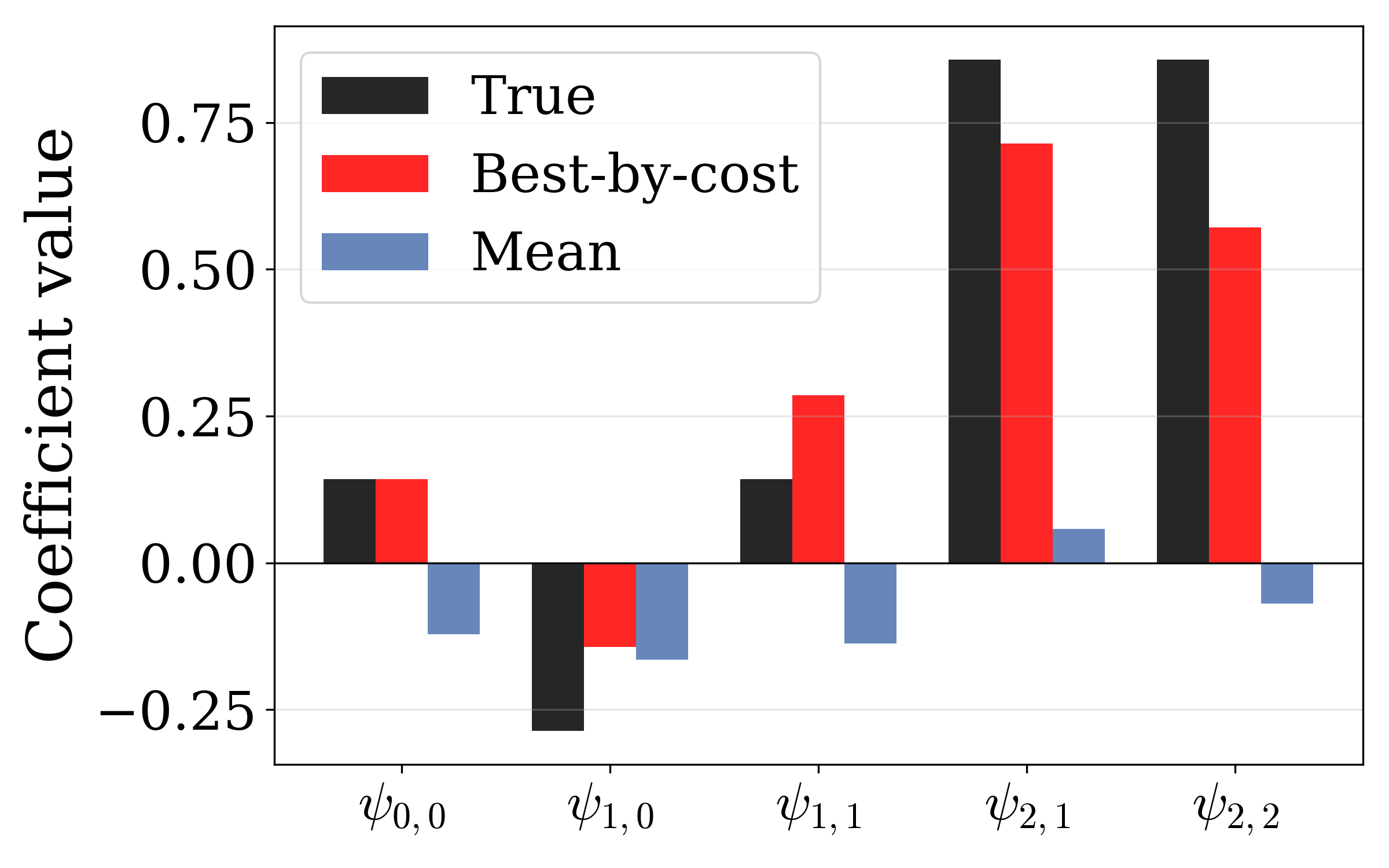}}
\subcaptionbox{4 bits per variable\label{fig:4bits2}}{%
\includegraphics[width=0.32\linewidth]{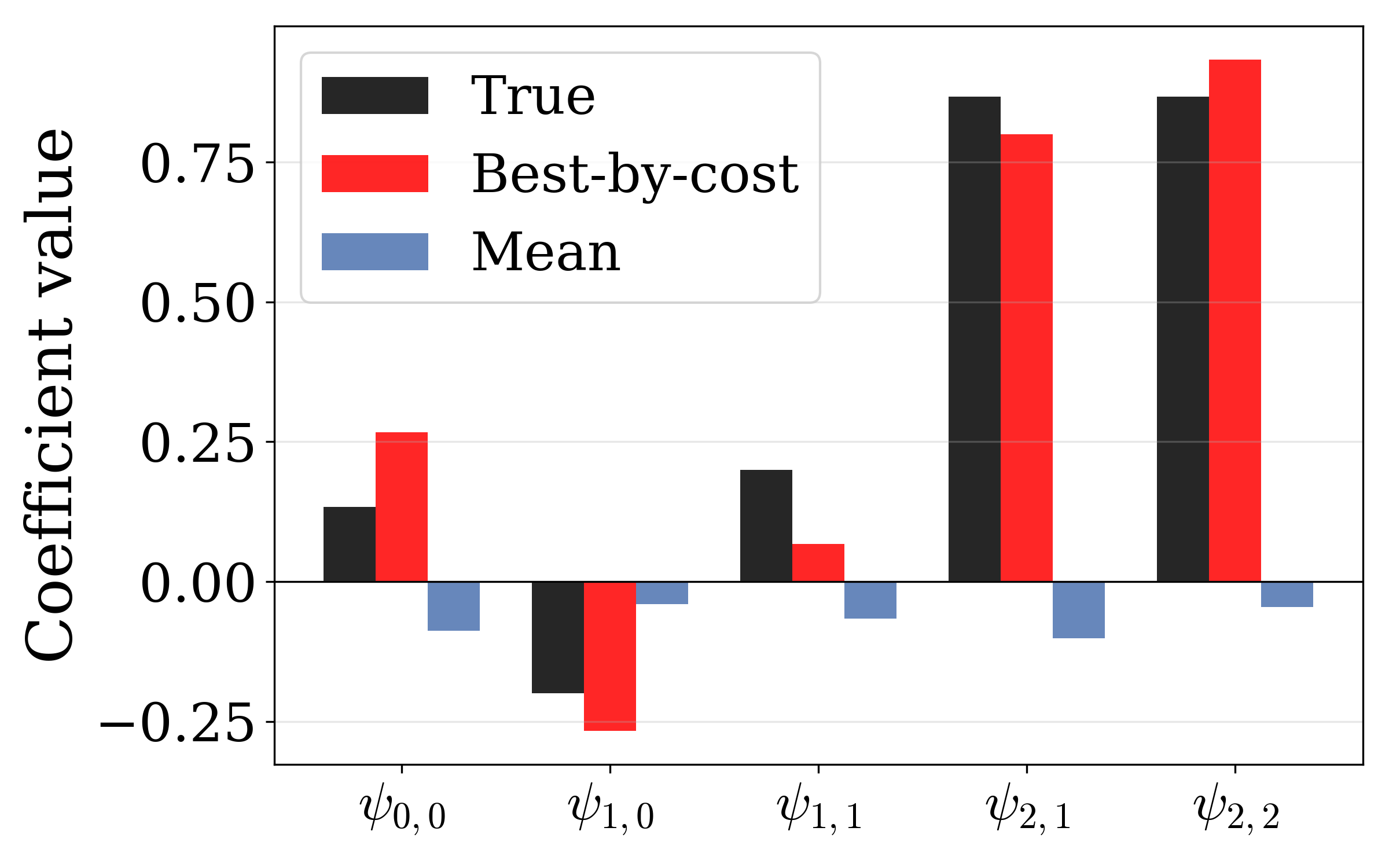}}

\end{tabular}

\caption{Recovered wavelet coefficients for varying bits per variable for solving the inverse problem in \eqref{eq:haar} on the VTT~Q50 Quantum computer with $T=4$, $p=3$ and $\varepsilon=0.03$. }
\label{fig:haarcoeffs}
\end{figure*}

\begin{table}[t]
\centering
\caption{Euclidean distance $\|\boldsymbol{x}_{\rm rec} - \boldsymbol{x}_{\rm true}\|$ and per-coefficient absolute errors $|x_{{\rm rec},\ell} - x_{{\rm true},\ell}|$ for the best-by-cost and mean estimates ($T=4$, $p=3$, $\varepsilon=0.03$, $N_s=10{,}000$ shots). Since the Haar basis matrix $\Phi$ has orthonormal columns, the signal reconstruction error satisfies $\|\boldsymbol{f}_{\rm rec} - \boldsymbol{f}_{\rm true}\| = \|\boldsymbol{x}_{\rm rec} - \boldsymbol{x}_{\rm true}\|$ and is therefore equal to the final column.}
\label{tab:haarcoeffs}
\begin{tabular}{llcccccr}
\toprule
 & & $\psi_{0,0}$ & $\psi_{1,0}$ & $\psi_{1,1}$
   & $\psi_{2,1}$ & $\psi_{2,2}$
   & $\|\cdot\|$ \\
\midrule
\multirow{2}{*}{$k=2$}
  & Best & $0.000$ & $0.000$ & $0.000$ & $0.000$ & $0.333$ & $0.333$ \\
  & Mean & $0.300$ & $0.605$ & $0.449$ & $0.712$ & $0.813$ & $1.352$ \\[4pt]
\multirow{2}{*}{$k=3$}
  & Best & $0.000$ & $0.143$ & $0.143$ & $0.143$ & $0.286$ & $0.378$ \\
  & Mean & $0.265$ & $0.121$ & $0.280$ & $0.799$ & $0.927$ & $1.289$ \\[4pt]
\multirow{2}{*}{$k=4$}
  & Best & $0.133$ & $0.067$ & $0.133$ & $0.067$ & $0.067$ & $0.221$ \\
  & Mean & $0.221$ & $0.160$ & $0.266$ & $0.968$ & $0.912$ & $1.383$ \\
\bottomrule
\end{tabular}
\end{table}

Second, the pattern of per-coefficient errors differs markedly
across settings.
For $k=2$, coefficients $\psi_{0,0}$ through $\psi_{2,1}$
(indices 0--3) are recovered with negligible error ($<0.001$),
while $\psi_{2,2}$ (index 4) has a large error of $0.333$,
equal to one grid spacing $\Delta x = 1/3$.
This single large error dominates the Euclidean distance of
$0.334$ and is directly responsible for the incorrect right-half
signal recovery seen in \ref{fig:haarsignal}a.
For $k=3$, the coefficient $\psi_{0,0}$ is recovered exactly, while the error is spread more evenly across coefficients
1--4, with none of these coefficient recovered exactly (errors of $0.143$
and $0.286$, corresponding to one and two grid spacings
$\Delta x=1/7\approx 0.143$).
For $k=4$, errors are the most evenly distributed (range
$0.066$--$0.134$) and the Euclidean distance $0.221$ is the
smallest of the three settings, consistent with the finer grid
resolution ($\Delta x=1/15\approx 0.067$) allowing a more accurate
representation.

Third, the mean coefficient estimates (blue bars in \ref{fig:haarcoeffs}) are close to zero for all settings and coefficients, with Euclidean distances to $\boldsymbol{x}_{\rm true}$ of $1.35$, $1.29$, and $1.38$ for $k=2,3,4$, respectively, substantially larger than the corresponding best-by-cost errors. The standard deviation of each marginal is approximately $0.5$ across all settings, consistent with a near-uniform distribution over the domain $[-1,1]$ induced by the split-variable representation $x_\ell=x_\ell^+-x_\ell^-$. Consequently, the mean provides a poor point estimate of the wavelet coefficients in this hardware setting.

\begin{figure*}[t]
\centering
\graphicspath{{Figures/}}

\setlength{\tabcolsep}{2pt}
\renewcommand{\arraystretch}{0}

\newcommand{\reconw}{0.24\textwidth}

\begin{tabular}{c}
\\[6pt]
\subcaptionbox{2 bits per variable\label{fig:2bits3}}{%
\includegraphics[width=0.32\linewidth]{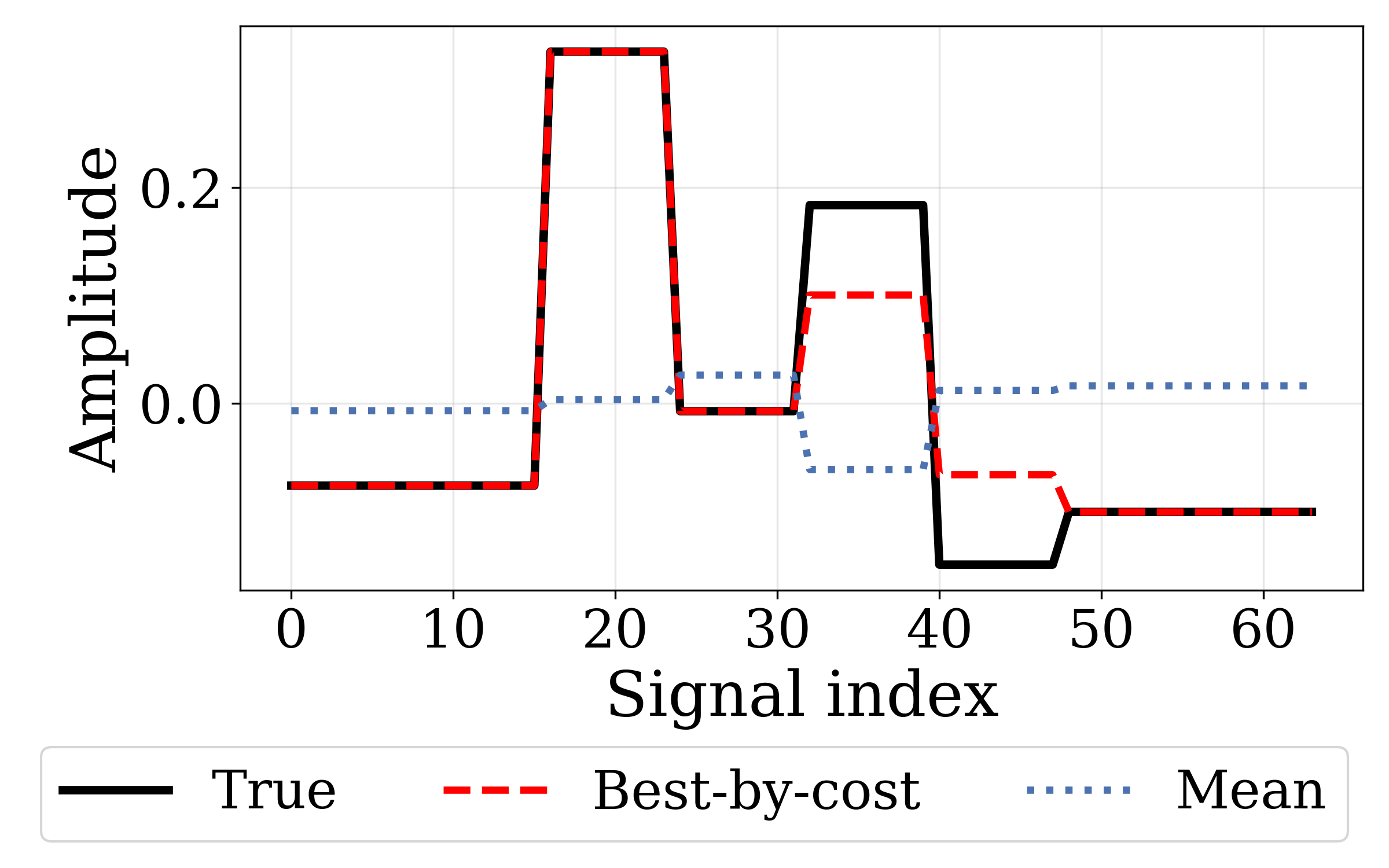}}
\subcaptionbox{3 bits per variable\label{fig:3bits3}}{%
\includegraphics[width=0.32\linewidth]{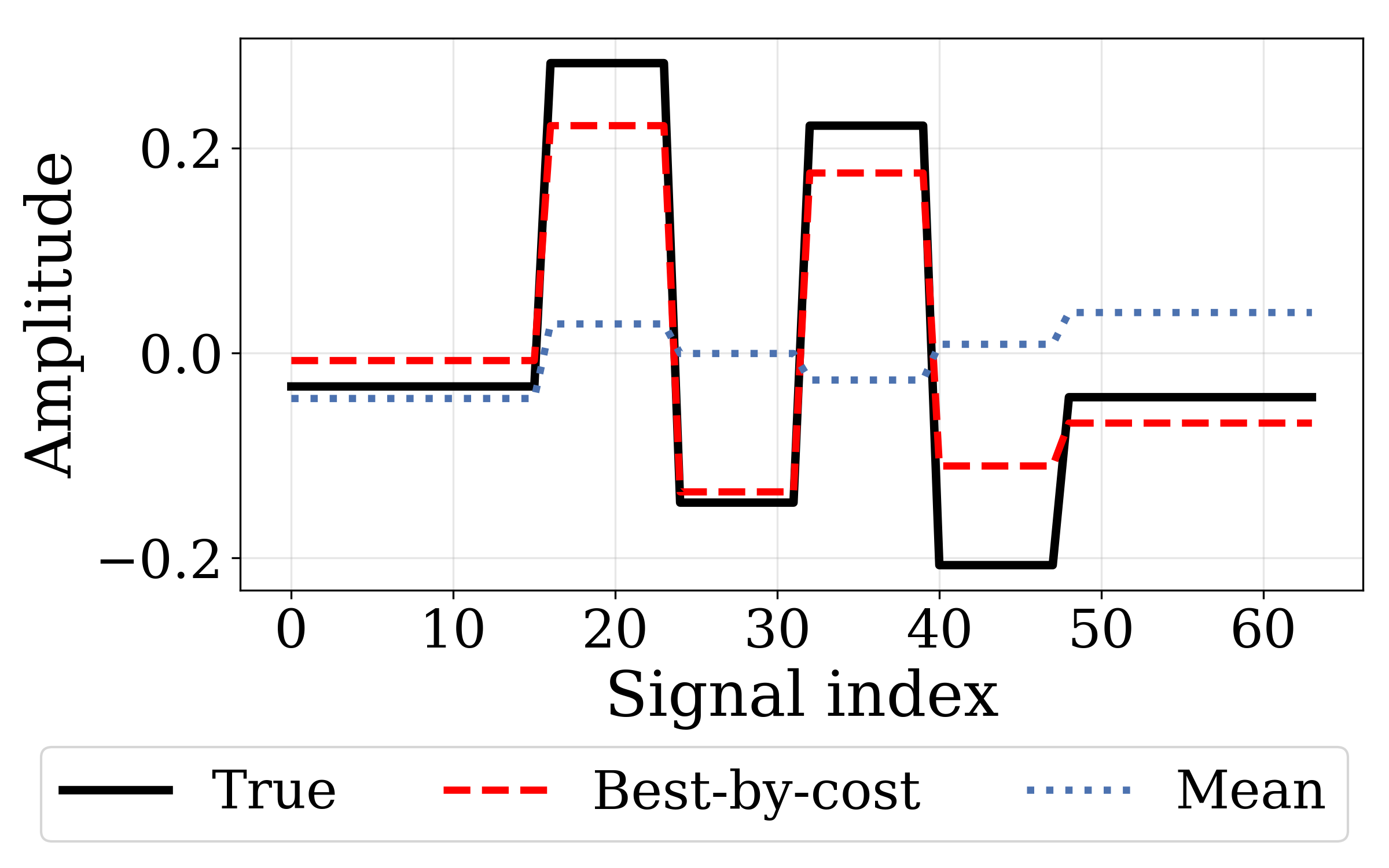}}
\subcaptionbox{4 bits per variable\label{fig:4bits3}}{%
\includegraphics[width=0.32\linewidth]{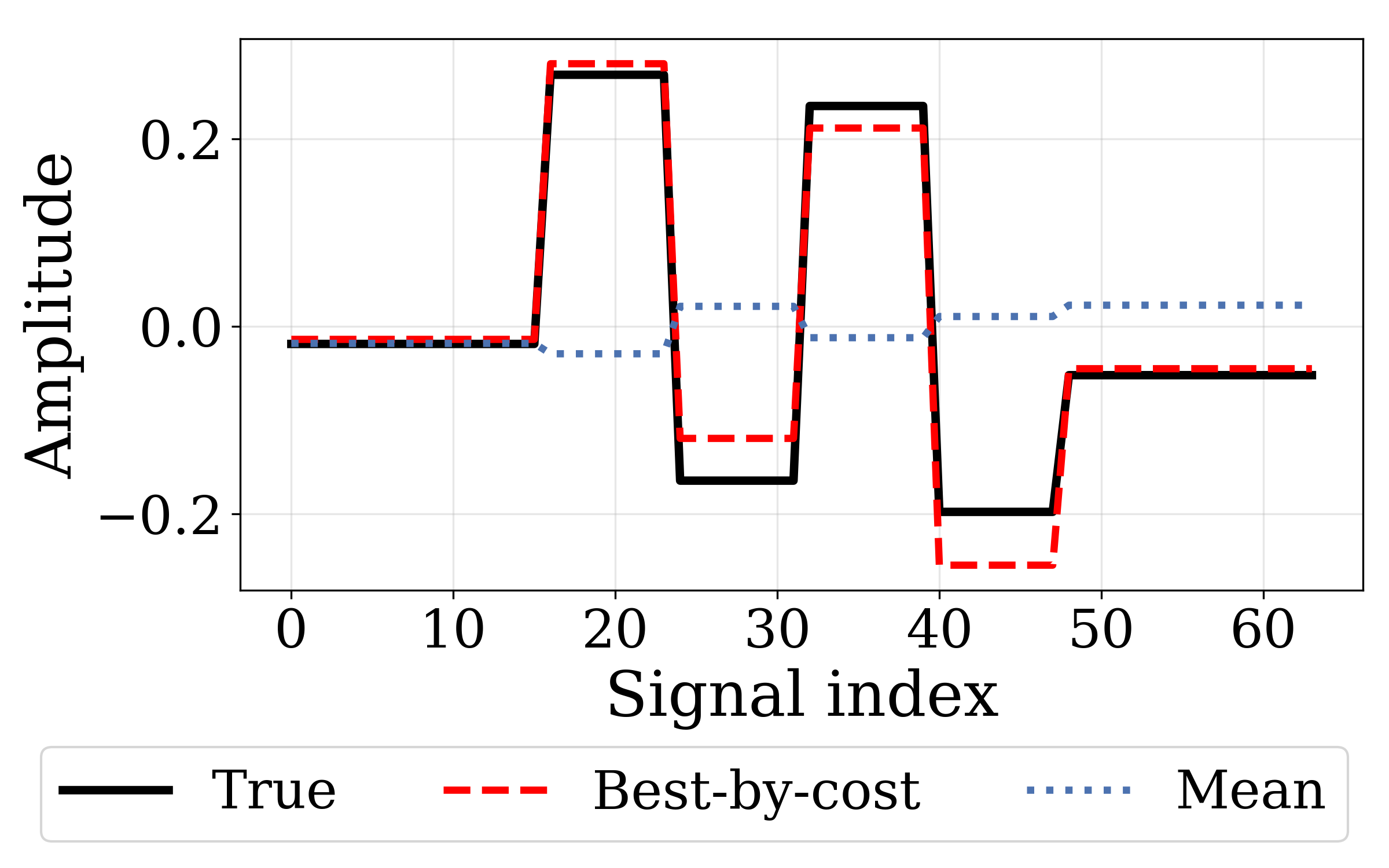}}
\end{tabular}
\caption{Recovered signal for varying bits per variable for solving the inverse problem in \eqref{eq:haar} with 5 wavelet coefficients on the VTT~Q50 Quantum computer with $T=4$, $p=3$ and $\varepsilon=0.03$. }
\label{fig:haarsignal}
\end{figure*}

\ref{fig:haarsignal} shows the true and recovered signals, with
signal reconstruction errors $\|\boldsymbol{f}_{\rm rec} -
\boldsymbol{f}_{\rm true}\|$ given in \ref{tab:haarcoeffs}. The true signal $f_{\rm true} = \Phi\boldsymbol{x}_{\rm true}$
is piecewise constant, with jumps occurring at the boundaries
between the positive and negative halves of each Haar basis function.


For $k=2$ (Fig.~\ref{fig:haarsignal}a), the best-by-cost reconstruction correctly recovers the regional mean and sign structure in all four quadrants, with a maximum pointwise error of $0.083$ at index $32$. Since the Haar basis is orthonormal, the signal reconstruction error equals the coefficient error, $\|\boldsymbol{f}_{\rm best}-\boldsymbol{f}_{\rm true}\|_2=0.333$ (see \ref{tab:haarcoeffs}), and is entirely due to the underestimation of the coefficient associated with $\psi_{2,2}$ by one grid spacing $\Delta x=1/3$. For $k=3$ (\ref{fig:haarsignal}b), the reconstruction again captures the correct sign pattern and overall structure, but with reduced amplitude due to the underestimation of the dominant coefficients associated with $\psi_{2,1}$ and $\psi_{2,2}$; the maximum pointwise error is $0.097$ at index $40$. For $k=4$ (\ref{fig:haarsignal}c), the reconstruction is the most accurate, with maximum pointwise error $0.057$ and coefficient error $\|\boldsymbol{f}_{\rm best}-\boldsymbol{f}_{\rm true}\|=0.221$, consistent with the smaller and more evenly distributed coefficient errors reported in \ref{tab:haarcoeffs}. In all three cases, the mean reconstruction remains close to zero throughout the domain, with errors $1.35$, $1.29$, and $1.38$ for $k=2,3,4$, respectively, confirming that the sample mean is not a useful estimator in this hardware setting.

\ref{tab:haarcoeffs} shows that the $k=4$ setting achieves the smallest best-by-cost error ($0.221$), followed by $k=2$ ($0.333$) and $k=3$ ($0.378$). These results should be interpreted with caution, however, since the parameters $T=4$, $p=3$, and $\varepsilon=0.03$ were tuned for the $k=4$ setting and then applied unchanged to $k=2$ and $k=3$. Consequently, the superior performance of $k=4$ is not unexpected, and the slightly poorer performance of $k=3$ relative to $k=2$ is more likely due to the mismatch between the fixed parameters and the corresponding QUBO formulation than to circuit complexity alone. A controlled study with parameters tuned separately for each value of $k$ would be required to assess the effect of encoding resolution. Nevertheless, the $k=2$ setting achieves competitive accuracy using only $20$ qubits, while the $k=4$ setting provides the best reconstruction, with all coefficient errors between one and two grid spacings ($\Delta x=1/15$). In contrast, the mean estimate is uninformative in all cases, with errors $1.352$, $1.289$, and $1.383$ for $k=2,3,4$, respectively, confirming that post-selection via the best-by-cost sample is the only effective estimation strategy on this hardware.

\section{Conclusions} \label{sec:conclusion}
We have presented a framework for formulating regularized linear inverse problems as quadratic unconstrained binary optimization (QUBO) problems and solving them using quantum optimization methods. The proposed approach accommodates both Tikhonov and sparsity-promoting regularization within a unified Hamiltonian formulation and extends naturally to variational inverse problems through wavelet-based representations. In addition, we introduced the notion of quantum sensitivity and derived theoretical bounds relating perturbations arising from approximate quantum evolution to stability properties of the underlying inverse problem. Numerical experiments demonstrated qualitative agreement between the theoretical analysis and the observed distributions of low-energy solutions in both simulated and physical quantum environments. While the empirical distributions do not fully reproduce the posterior geometry predicted by the continuous theory, they exhibit nontrivial directional structure and capture salient features of the underlying inverse problem. Results obtained on quantum hardware further support the practical viability of the proposed framework and its ability to identify accurate low-energy reconstructions through post-selection.

At the same time, the experiments highlight important limitations of current quantum devices. While the simulated and physical quantum systems exhibit qualitatively similar behavior, noticeable discrepancies remain in the resulting solution distributions. In particular, the empirical distributions observed on hardware differ from their simulated counterparts, and the estimated posterior means do not fully agree with either the theoretical predictions or the simulation results.

Several directions for future work remain open. Most notably, the present study is restricted to linear inverse problems with QUBO-based Hamiltonians. However, the underlying quantum optimization framework is not inherently tied to quadratic or even differentiable objective functions. Since the adiabatic evolution paradigm depends primarily on the specification of a suitable Hamiltonian rather than on particular analytical properties of the objective, extending these ideas to nonlinear inverse problems and more general variational formulations represents a promising research direction. Such extensions may enable the exploration of inverse problem methodologies that are difficult to realize within existing classical optimization frameworks and further clarify the role of quantum computation in scientific inference.

\section*{Acknowledgments}
This work was funded by the  Flagship of Advanced Mathematics for Sensing, Imaging and Modelling (grant no. 359186). S. Rautio also acknowledges support from the Väisälä Fund through the Finnish Academy of Science and Letters. A. Hauptmann is also supported by Research Council of Finland under grant numbers 338408, 370528.
B.~M.~Afkham is also supported by Research Council of Finland under grant numbers 371523.

The authors acknowledge the computational resources and support provided by CSC – IT Center for Science (Finland) under Finnish Quantum-Computing Infrastructure projects 462001253 and 462001377.

\bibliographystyle{plain}
\bibliography{references}
\end{document}